\documentclass[12pt]{article}

\usepackage{pgfplots}
\usepackage{times}
\usepackage{amsmath,amsfonts, amstext,amssymb,amsbsy,amsopn,amsthm}
\usepackage{dsfont}
\usepackage{esint}
\usepackage{graphicx}   
\usepackage{hyperref}
\usepackage[all]{xy}
\usepackage{color}
\usepackage{tikz}
\usepackage{lineno,hyperref}
\usetikzlibrary{intersections}

\newtheorem{mthm}{Theorem}
\newtheorem{theorem}{Theorem}[section]
\newtheorem{mrem}{Remark}
\newtheorem{definition}{Definition}[section]
\newtheorem{mlem}{Lemma}
\newtheorem{lemma}[definition]{Lemma}
\newtheorem{proposition}[definition]{Proposition}

\theoremstyle{remark}
\newtheorem{remark}[definition]{Remark}
\numberwithin{equation}{section}

\newcommand{\R}{\mathbb{R}}

\newcommand{\lap}{\mbox{$\triangle$}}

\newcommand{\fr}{\displaystyle\frac}
\newcommand{\jf}{\displaystyle\int}

\newcommand{\mb}{\mbox}

\newcommand{\be}{\begin{equation}}
\newcommand{\ee}{\end{equation}}
\newcommand{\bee}{\begin{equation*}}
\newcommand{\eee}{\end{equation*}}

\title{ Interior regularity for fully fractional evolution equations and a priori estimates on unbounded domains}

\author{Wenxiong Chen,
Yahong Guo, Congming Li}

\begin{document}

\maketitle

\begin{abstract}

In this paper, we study the fully fractional heat equation involving the master operator:
$$
(\partial_t -\Delta)^{s} u(x,t) = f(x,t)\ \ \mbox{in}\ \R^n\times\R.
$$

We derive interior H\"{o}lder and Schauder estimates for non-negative solutions that depend solely on local data. Owing to the non-locality of the operator, earlier work required {\em uniform global bounds} to control higher local norms--an assumption that fails during blow-up and rescaling on {\em unbounded domains}, where rescaled solutions may develop arbitrarily large tails.

After introducing some new ideas and novel techniques, such as using a {\em directional perturbation average} to derive a key estimate for the fully fractional heat kernel, we are able to employ {\em local bounds} of solutions to replace global bounds to control their higher local norms.

Based on these new estimates, we are capable to apply the blow-up and rescaling arguments to establish a priori estimates
for solutions to a broader class of nonlocal equations in unbounded domains, such as
$$(\partial_t -\Delta)^{s} u(x,t) =  b(x,t) |\nabla_x u (x,t)|^q + f(x, u(x,t))\ \ \mbox{in}\ \ \R^n\times\R.$$
Under appropriate conditions, we prove that all nonnegative solutions, along with their spatial gradients, are
uniformly bounded.

 \end{abstract}

 \bigskip

\textbf{Mathematics subject classification (2020): }  35R11; 35K99; 47G30.
\bigskip

\textbf{Keywords:}  master equations, interior regularity estimates, blowing up and re-scaling, a priori estimates     \\
\medskip

\section{Introduction}

In this paper, we first establish new regularity estimates for nonnegative solutions for the following master equations
\begin{equation}\label{1.0}
(\partial_t -\Delta)^{s} u(x,t) =  f(x,t)\ \ \mbox{in}\ \ \R^n\times\R.
\end{equation}
Using these regularity results, we then apply blow-up and rescaling analysis to derive a priori estimates for solutions to the nonlinear equations
\begin{equation}\label{1.1}
(\partial_t -\Delta)^{s} u(x,t) =  b(x,t) |\nabla_x u (x,t)|^q + f(x, u(x,t))\ \ \mbox{in}\ \ \R^n\times\R.
\end{equation}

Here the fully fractional heat operator $(\partial_t-\Delta)^s$, initially introduced by M. Riesz in \cite{Riesz}, is a nonlocal pseudo-differential operator of order $2s$ in the spatial variables and of order $s$ in the time variable. It is defined by the following singular integral
\begin{equation}\label{nonlocaloper}
(\partial_t-\Delta)^s u(x,t)
:=C_{n,s}\int_{-\infty}^{t}\int_{\mathbb{R}^n}
  \frac{u(x,t)-u(y,\tau)}{(t-\tau)^{\frac{n}{2}+1+s}}e^{-\frac{|x-y|^2}{4(t-\tau)}}\operatorname{d}\!y\operatorname{d}\!\tau,
\end{equation}
where $0<s<1$, the integral in $y$ is in the sense of Cauchy principal value, and
the normalization constant is given by $$C_{n,s}=\frac{1}{(4\pi)^{\frac{n}{2}}|\Gamma(-s)|}$$
with $\Gamma(\cdot)$ denoting the Gamma function.

This operator is inherently nonlocal in both space and time, as the value of $(\partial_t-\Delta)^s u$ at any point $(x,t)$ depends on the values of $u$ over the entire spatial domain $\mathbb{R}^n$ and all past times before $t$.

We say that $u$ is a classical entire solution of \eqref{1.0} if
 $$u(x,t)\in C^{2s+\epsilon,s+\epsilon}_{x,\, t,\, {\rm loc}}(\mathbb{R}^n\times\mathbb{R}) \cap \mathcal{L}(\mathbb{R}^n\times\mathbb{R})$$
for some $\varepsilon >0$.
This condition ensures that the singular integral in \eqref{nonlocaloper} converges. Here,
the slowly increasing function space $\mathcal{L}(\mathbb{R}^n\times\mathbb{R})$ is defined by
$$ \mathcal{L}(\mathbb{R}^n\times\mathbb{R}):=\left\{u(x,t) \in L^1_{\rm loc} (\mathbb{R}^n\times\mathbb{R}) \mid \int_{-\infty}^t \int_{\mathbb{R}^n} \frac{|u(x,\tau)|e^{-\frac{|x|^2}{4(t-\tau)}}}{1+(t-\tau)^{\frac{n}{2}+1+s}}\operatorname{d}\!x\operatorname{d}\!\tau<\infty,\,\, \forall \,t\in\mathbb{R}\right\}.$$
The definition of the local parabolic H\"{o}lder space $C^{2s+\epsilon,s+\epsilon}_{x,\, t,\, {\rm loc}}(\mathbb{R}^n\times\mathbb{R})$ will be specified in Section 2.

It is worth noting that the fractional powers of heat operator $(\partial_t-\Delta)^s$
reduces to the classical heat operator $\partial_t-\Delta$ as $s\rightarrow 1$ (cf. \cite{FNW}). Moreover,
when the space-time nonlocal operator $(\partial_t-\Delta)^s$ is applied to a function $u$ that depends only on the spatial variable $x$, it simplifies to
 \begin{equation*}
   (\partial_t-\Delta)^s u(x)=(-\Delta)^s u(x),
 \end{equation*}
where $(-\Delta)^s$ denotes the well-known fractional Laplacian, defined by
$$ (-\Delta)^s u(x)=C_{n, s}PV \int_{\mathbb{R}^n}\frac{u(x)-u(y)}{|x-y|^{n+2s}}dy.$$

This operator is of considerable interest due to its wide-ranging applications in various scientific fields, including  physics, chemistry, and biology. It arises in context such as anomalous diffusion, quasi-geostrophic flows,
thin obstacle problem, phase transitions, crystal dislocation, flame propagation, conservation
laws, multiple scattering, minimal surfaces, optimization, turbulence models, water waves,
molecular dynamics, and image processing ( see \cite{AB, BG, CV, GO} and references therein).

Furhtermore, these operators play crucial roles in probability
and finance \cite{Be} \cite{CT} \cite{RSM}. In particular, the fractional Laplacians can be interpreted as the
infinitesimal generator of a stable L\'{e}vy process \cite{Be}.

In recent decades, followed the pioneer work of Caffarelli and Silvestre \cite{CS}, significant attention has been devoted to the study of solutions to fractional elliptic equations, leading to a wealth of important results. For further details,  interested readers may consult \cite{AGHW, AWY, CL, CLL, CLL1, CLZ,  CZhu, CW, CWW, DLL, FLS, KMW, LLW, LW, LWX, LZ, LZ1}. For late developments concerning qualitative properties of various fractional parabolic equations, we referred to \cite{CG, CM, GMZ2, WuC}  and references therein.

In the special case when $u=u(t)$, we have:
 \begin{equation*}
   (\partial_t-\Delta)^s u(t)=\partial_t^s u(t),
 \end{equation*}
where $\partial_t^s$ denotes the Marchaud fractional derivative of order $s$, defined as \begin{equation}\label{1.00}
\partial^s_t u(t)=C_s \jf_{-\infty}^t\fr{u(t)-u(\tau)}{(t-\tau)^{1+s}}d\tau.
\end{equation}
This derivative arises in various physical phenomena, such as
particle systems with sticking and trapping effects, magneto-thermoelastic heat conduction, plasma turbulence and more (cf. \cite{ACV,ACV1, DCL1, DCL2, EE}).

The space-time nonlocal equations \eqref{1.0} and \eqref{1.1} are also significant in many physical and biological applications, including anomalous diffusion \cite{KBS}, chaotic dynamics \cite{Z}, biological invasions \cite{BRR}. In the financial domain, these equations are valuable for modeling situations where the waiting time between transactions correlates with subsequent price jumps (cf. \cite{RSM}). One notable application of the master equation is its role in describing continuous-time random walks, where $u$ represent the distribution of particles undergoing random jumps alongside random time lags (cf. \cite{MK}).
\medskip

\leftline{\bf \large The regularity estimates}
\medskip

Let
$$Q_r (x^o, t_o) = \{(x,t) \in \R^n \times \R \mid |x-x^o| < r, \; |t-t_o| < r^2 \}$$
denote the parabolic cylinder of size $r$ centered at $(x^o, t_o)$.

We use the standard notation for parabolic H\"{o}lder space. A function $u(x,t)$ is said to belong to $C_{x,t}^{\alpha, \beta}$ if $u$ is $C^\alpha$ in the spatial variable $x$ and $C^\beta$ in the time variable $t$. If $\alpha <1$, $C^\alpha$ is the usual H\"{o}lder space.
If $1<\alpha <2$, $C^\alpha$ denotes the space of bounded functions whose first derivatives belong to $C^{\alpha -1}$, and so on.  For more precise definitions, please see Section 2.

Regarding regularity estimates for master equation \eqref{1.0}, Stinga and Torrea \cite{ST}, among other results, established the following local H\"{o}lder estimate ($0<s<1/2$ case)
\begin{equation}  \label{ST1}
\|u\|_{C_{x,t}^{2s, s}(Q_1(0,0))} \leq C \left( \|f\|_{L^\infty(Q_2(0,0))} + \|u\|_{L^\infty (\mathbf{\R^n \times (-\infty, 4)})}\right),
\end{equation}
and Schauder estimate:
\begin{equation}  \label{ST2}
\|u\|_{C_{x,t}^{2s + \alpha, s + \alpha/2}(Q_1(0,0))} \leq C \left( \|f\|_{C_{x,t}^{\alpha, \alpha/2}(Q_2(0,0))} + \|u\|_{L^\infty (\bf{\R^n \times (-\infty, 4)})}\right) .
\end{equation}

A similar H\"{o}lder estimate was obtained by Caffarelli and Silvestre in \cite{CS1}.

In order to circumvent the nonlocality of the master operator, the authors in \cite{ST} relied on the extension method introduced by Caffarelli and Silvestre \cite{CS}.

However, due to the nonlocal nature of the fully fractional heat operator, these results appear significantly weaker than those obtained for the
 classical heat equation
$$ (\partial_t - \Delta) u(x,t) = f(x,t),  $$
where the global norm $\|u\|_{L^\infty (\R^n \times (-\infty, 4))}$ on the right hand sides of \eqref{ST1} and \eqref{ST2} can  be replaced by a weaker local norm $\|u\|_{L^\infty (Q_2(0,0))}$.

 More critically, the presence of such global norms on the right hand sides of \eqref{ST1} and \eqref{ST2} makes it impossible to apply the regularity estimates  to the well-known blow-up  and rescaling arguments to establish a priori estimates for fractional equations on unbounded domains.

As an illustration, let us compare a simple problem of local nature
\begin{equation}
(\partial_t - \Delta) u(x,t) = u^p(x,t), \;\; (x,t) \in \R^n \times \R,
\label{1001}
\end{equation}
with its nonlocal counter part:
\begin{equation}
(\partial_t - \Delta)^s u(x,t) = u^p(x,t), \;\; (x,t) \in \R^n \times \R.
\label{1002}
\end{equation}

Here the right hand side can actually be replaced by a more general nonlinearity $f(x, u, \nabla u)$.

To obtain an a priori estimate--namely, to show that all nonnegative solutions are uniformly bounded--one typically employs a blow-up and rescaling argument, which proceeds as follows.

If the solutions are not uniformly bounded, then there exist a sequence of solutions $\{u_k\}$ and a sequence of points $\{(x^k, t_k)\}$,
such that $u_k(x^k, t_k) \to \infty$. Note that this $u_k(x^k,t_k)$ may not be chosen to be the maximum of $u_k$ in $\R^n \times \R$, because each $u_k$ may be unbounded. However, similar to the well-known {\em doubling lemma}, one can find a nearby point $(\bar{x}^k, \bar{t}_k)$ and some small $\lambda_k$, such that
$$ u_k(x,t) \leq C \, u_k (\bar{x}^k, \bar{t}_k) \equiv C \, m_k, \;\; \forall \, (x,t) \in Q_{\lambda_k R} (\bar{x}^k, \bar{t}_k).$$
Here $R>0$ can be chosen arbitrary large.

Upon re-scaling
$$ v_k (x,t) = \frac{1}{m_k} u_k (\lambda_k x + x^k, \lambda^2_k t + t_k),$$
the new sequence of functions $v_k$ becomes bounded and satisfy similar equations on certain parabolic cylinders $Q_R (0,0)$. However, outside $Q_R(0,0)$, no control over $v_k$ is available, and it may even become unbounded.

{\em For equation \eqref{1001} of local nature}:

The boundedness of $\{v_k\}$ within $Q_R(0,0)$  is  {\em sufficient} to ensure boundedness of its higher derivatives. This allows one to extract a subsequence of  $\{v_k\}$ that converges to a function $v$, which is a bounded solution of the same equation. By using known results on the nonexistence of bounded positive solutions (cf. Quittner \cite{Q}), one can derive a contradiction and thereby establish  the desired a priori estimate. Specifically, Quittner showed that no {\em bounded} positive solutions for \eqref{1001} when $1<p< \frac{n+2}{n-2}$. Using  this result in combination with the blow-up and rescaling arguments, the non-existence of {\em unbounded} positive solutions can also be deduced.

{\em For equation   \eqref{1002} of nonlocal nature}:

In contrast, the available regularity estimates, such as   \eqref{ST1} and \eqref{ST2}, reveal that the boundedness of $\{v_k\}$ on $Q_R(0,0)$ is {\em insufficient} to guarantee the bound on its higher norms. Instead, to bound these norms, $\{v_k\}$ {\em must be bounded globally across the entire
space.} This condition cannot be satisfied since no  information about the behavior of $v_k$ beyond $Q_R(0,0)$ is available during the rescaling process. This presents  a substantial difficulty when attempting to use blow-up and rescaling arguments for nonlocal equations on unbounded domains.

These challenges lead to the following natural question:

{\em Can one replace  the  global norm $\|u\|_{L^\infty (\R^n \times (-\infty, 4))}$  with the weaker local norm $\|u\|_{L^\infty (Q_2(0,0))}$ to bound the higher norm of $u$ under certain conditions? }
\smallskip

Addressing this challenging question is one of the main objectives of the present paper. Rather than relying on the {\em extension method}, we approach this issue directly through integral representations of solutions . After introducing some brand new ideas and through careful analysis, we answer this question affirmatively for nonnegative solutions $u(x,t)$ to master equation \eqref{1.0} with  $f(x,t)\geq 0$. We establish the following regularity estimates:
\begin{mthm} (H\"{o}lder estimates) \label{mthmu1} Assume that both $f(x,t)$ and $u(x,t)$ are bounded in $Q_2(0,0)$. Then there exists a positive constant $C$, such that
$$
\|u\|_{C^{2s, s}_{x, t} (Q_1(0,0))} \leq C \, \left( \|f\|_{L^\infty (Q_2(0,0))} + \|u\|_{L^\infty (\mathbf{Q_2(0,0)})}\right), \;\; \mbox{ if } s \neq 1/2.
$$
$$
\|u\|_{C^{log L, s}_{x, t} (Q_1(0,0))} \leq C \, \left( \|f\|_{L^\infty (Q_2(0,0))} + \|u\|_{L^\infty (\mathbf{Q_2(0,0)})} \right), \;\; \mbox{ if } s = 1/2.
$$
\end{mthm}

\begin{mthm} (Schauder estimates) \label{mthmu2} Assume that $f \in C^{\alpha, \alpha/2}_{x,t}  (Q_2(0,0))$ for some $0 < \alpha < 1$. Then there exists a positive constant $C$, such that
$$
\|u\|_{C^{2s+ \alpha, s+\alpha/2}_{x, t} (Q_1(0,0))} \leq C \, \left(\|f\|_{C^{\alpha, \alpha/2}_{x,t}  (Q_2(0,0))} + \|u\|_{L^\infty (\mathbf{Q_2(0,0)})} \right), \; \; \; \mbox{ if } \; 2s +\alpha < 1 \mbox{ or } 1< 2s +\alpha < 2;
$$
$$
\|u\|_{C^{\log L, s+\alpha/2}_{x, t} (Q_1(0,0))} \leq C \, \left(\|f\|_{C^{\alpha, \alpha/2}_{x,t}  (Q_2(0,0))} + \|u\|_{L^\infty (\mathbf{Q_2(0,0)})} \right), \; \; \; \mbox{ if } \; 2s + \alpha = 1;
$$
$$
\|u\|_{C^{1+\log L, \log L}_{x, t} (Q_1(0,0))} \leq C \, \left(\|f\|_{C^{\alpha, \alpha/2}_{x,t}  (Q_2(0,0))} + \|u\|_{L^\infty (\mathbf{Q_2(0,0)})} \right), \;\;\; \mbox{ if } \; 2s + \alpha = 2.
$$
\end{mthm}

Here $C^{Log \, L}$ is the standard ``Log-Lipschitz'' space and one can find the precise definition right before the statement of Theorem
\ref{mthmw1}.

We believe that this is the first instance where the local higher norms of solutions to a {\em nonlocal} equation can be effectively controlled by their local norms. This can be viewed as a breakthrough in the area, which makes it possible to employ the blow-up and re-scaling arguments to
establish a priori estimates for solutions to a family of similar nonlocal equations in unbounded domains, as the readers will see in a moments later.

\begin{mrem}
From the proofs in Section 3 and 4, it is evident that similar arguments apply to higher derivatives estimates. Here, we present only  the results necessary for our blow-up and rescaling analysis.
\end{mrem}

To obtain such regularity estimates, we first prove the equivalence between the pseudo differential equation and an integral equation.
\begin{mthm} \label{mthm1}
Assume that $f(x,t) \geq 0$ and  $u(x,t)$ is a nonnegative classical solution of \eqref{1.0}, then it satisfies the integral equation
\begin{equation} \label{inteq}
u(x,t) = c + \int_{-\infty}^t \int_{\mathbb{R}^n} f(y,\tau) G(x-y, t-\tau) dy d\tau,
\end{equation}
where
$$G(x-y, t-\tau) = \frac{C_{n,s}}{(t-\tau)^{n/2+1-s}} e^{-\frac{|x-y|^2}{4(t-\tau)}}$$
is the fundamental solution, or the Green's function associated with the master operator $(\partial_t - \lap)^s$, and $c$ is a nonnegative constant.
\end{mthm}

Since the constant $c$ in \eqref{inteq} does not influence either the H\"{o}lder norm or the derivatives of $u$, we may, without loss of generality, assume for the remainder of the paper that $c=0$ and hence
\begin{equation} \label{inteq1}
u(x,t) = \int_{-\infty}^t \int_{\mathbb{R}^n} f(y,\tau) G(x-y, t-\tau) dy d\tau,
\end{equation}

For simplicity of notation, we denote
$$ Q = Q_2(0,0), \mbox{ and } \; \; \tilde{Q} = Q_1(0,0) \subset \subset Q.$$

We split $u(x,t)$ into two parts $u(x,t)=v(x,t) + w(x,t)$ as  in the following.

Write
$$ f_Q (x,t) = \left\{ \begin{array}{ll} f(x,t), & (x,t) \in Q \\
0, & (x, t) \in Q^c.
\end{array}
\right.
$$

Denote
$$w(x,t) = \int_{-\infty}^t \int_{\mathbb{R}^n} f_Q(y,\tau) G(x-y, t-\tau) dy d\tau$$
Then $w(x,t)$ satisfies the nonhomogeneous equation in $Q$:
$$(\partial_t -\Delta)^s w(x,t) = f(x,t), \;\; \forall \, (x,t) \in Q.$$
While the homogeneous part can be expressed by
$$v(x,t):=u(x,t)-w(x,t)=\int_{-\infty}^t \int_{\mathbb{R}^n} f_{Q^c}(y,\tau) G(x-y, t-\tau) dy d\tau$$
with $f_{Q^c}:=f-f_Q$. Apparently, $v$ is a solution of the homogeneous equation in $Q$:
$$(\partial_t - \Delta)^s v(x,t) = 0, \;\; \forall \, (x,t) \in Q.$$
\smallskip

We estimate $v(x,t)$ and $w(x,t)$ separately. We first obtain
\begin{mthm} \label{mthmv} Assume that $v$ is bounded in $Q$, then for any $\alpha \in (0,1)$,
$$
\|v\|_{C_{x,t}^{2, \alpha} (\tilde{Q})} \leq C \|v\|_{L^\infty (Q)}.
$$
\end{mthm}

\begin{mrem}
Using similar arguments as in our proof (see Section 3), it can be shown that the $C_{x,t}^{2, \alpha}$ norm on the left hand side can be replaced by $C_{x,t}^{k, \alpha}$ norm for any integer $k$. Nonetheless, this $C_{x,t}^{2, \alpha}$ norm is sufficient for our blow-up and rescaling analysis.
\end{mrem}

\leftline{\bf New ideas involved in the proof}
\smallskip

To estimate the derivatives of $v(x,t)$, say $v_{x_i}$,  in terms of $v(x,t)$ itself, we encounter a significant challenge.
From the definition
$$v(x,t) = C_{n,s}\int_{-\infty}^t \int_{\mathbb{R}^n}  \frac{f_{Q^c}(y,\tau)}{(t-\tau)^{n/2+1-s}} e^{-\frac{|x-y|^2}{4(t-\tau)}} dy d\tau, $$
it follows that
$$v_{x_i}(x,t)=C_{n,s}\int_{-\infty}^t \int_{\mathbb{R}^n}  \frac{f_{Q^c}(y,\tau)}{(t-\tau)^{n/2+1-s}} \frac{x_i-y_i}{2(t-\tau)}e^{-\frac{|x-y|^2}{4(t-\tau)}} dy d\tau.$$

Comparing these two expressions, it appears that, in order to control $v_{x_i}$ in terms of $v$, we need to establish the inequality
\begin{equation}\label{est-fuc1}\frac{x_i-y_i}{2(t-\tau)}e^{-\frac{|x-y|^2}{4(t-\tau)}}\leq Ce^{-\frac{|x-y|^2}{4(t-\tau)}}.\end{equation}
However, this seems impossible because the extra term $\frac{x_i-y_i}{2(t-\tau)}$ on the left hand side is clearly unbounded. To overcome this difficulty, we introduce an innovative idea. Instead of attempting to control $|v_{x_i}(x,t)|$ directly in terms of $v(x,t)$ at the same point $(x,t)$, we find finitely many nearby points $(x^j,t)$ to accomplish this task:
$$ |v_{x_i}(x,t)| \leq C \sum_j^{2^n} v(x^j, t).$$

To illustrate this approach, let us consider the case of two dimensions ($n=2$).

For each fixed  $x\in B_1(0)$, we take $x$ as the center and divide the plane $\R^2$ into four quadrants $I_1, I_2, I_3, I_4.$
On each quadrant $I_j$, choose the point $x^j$ as the intersection of $\partial B_{1/\sqrt{2}}(x)$ and the diagonal of $I_j, \,  j=1, \cdots, 4$:
$$ x^j = x + \eta_j, \mbox{ with } \eta_1 = \frac{1}{\sqrt{2}}(\cos \frac{\pi}{4}, \sin \frac{\pi}{4}), \eta_2 = \frac{1}{\sqrt{2}} (\cos \frac{3\pi}{4}, \sin \frac{3\pi}{4}), \cdots.$$
We will verify that for all $y\in B_2^c(0)\cap I_j$ it holds
$$|y-x|^2 \geq |y-x^j|^2 + {\bf \frac{1}{2} |y-x|}.$$

This perturbation produced a term ${\bf  e^{\frac{c|x-y|}{t-\tau}}}$ on the right hand side, which can be used exactly to control the extra term ${\bf \frac{|x-y|}{t-\tau}}$ on the left hand side  of \eqref{est-fuc1}.

 Using this inequality, we are able to control the left hand side of \eqref{est-fuc1} by the right hand side at point $x^j$:
 $$
\frac{|y-x|}{t-\tau} e^{-\frac{|y-x|^2}{4(t-\tau)}} \leq C e^{-\frac{|y-x^j|^2}{4(t-\tau)}}, \;\; \mbox{ for all }  y \in B_2^c (0) \cap I_j.
$$

Consequently
for all $(x, t) \in \tilde{Q}$, we obtain
$$|v_{x_i} (x, t)| \leq C\sum_j^4  v(x^j,t) \leq C \|v\|_{L^\infty (Q)}. $$

For further details, please refer to Figure 1 and arguments in Section 3.

The above approach can be interpreted as a {\em directional perturbation average}. At a given point $x$, it is impossible to control the derivative of $v$ solely by the value of $v$ itself. However, after making directional perturbations from $x$ to $x^j = x + \eta_j$ for $ j=1,2, \cdots, 2^n$ along directions $\eta_j$, and then taking the average, we are able to realize such a control. Such an estimate is essentially performed on the fully fractional heat kernel $G(x-y, t-\tau)$. For convenience of future applications, we summarized this as
follows:
\begin{mlem}[ Fully  fractional heat kernel estimates ]\label{key0} In each quadrant $I_j$ of $R^n$, there is a vector $\eta_j\in\partial B_{\frac{1}{\sqrt{n}}}(0)$ along the diagonal direction such that
\begin{equation}\label{key1} G(y,\tau)\leq {\bf{e^{-\frac{|y|}{n\tau}}}}\sum\limits_{j=1}^{2^n}G(y+{\eta}_j,\tau),  \forall \tau>0, y\in B_1^c(0).\end{equation}
Furthermore,   for each $\tau>0, y\in B_1^c(0)$, it holds
\begin{equation} D_y^\alpha  G(y,\tau)\leq C\sum\limits_{j=1}^{2^n}G(y+{\eta}_j,\tau), \forall \, |\alpha|=1 \, or \, 2,\end{equation}
and
 \begin{equation}  \partial_\tau G(y,\tau)\leq C\sum\limits_{j=1}^{2^n}G(y+{\eta}_j,\tau),\end{equation}
here $C=C(n)$ is a universal constant.
\end{mlem}

 This lemma will be validated in the proof of Lemma \ref{thm2.1.1}.
 \smallskip

 We believe that this important heat kernel estimate will become a powerful tool in establishing regularity estimates, with potential applications to a variety of related problems.
\medskip

To derive H\"{o}lder estimates for $w(x,t)$, we assume that $f \in L^\infty (Q)$, with the results depending on whether $2s=1$ or not.
In the case $2s=1$, $w$ belongs to the ``Log-Lipshitz'' space, which is defined as follows:
Let $g(x)$ be a function of $x$. We say that $g$ belongs to the ``Log-Lipshitz" space
 $C^{\log L}(\Omega)$ if
\begin{eqnarray*}
\|g\|_{C^{\log L}(\Omega)}=\|g\|_{L^\infty (\Omega)}+ \sup_{x, y\in \Omega, x\neq y} \frac{|g(x)-g(y)|}{|x-y|  |\log \min\{|x-\bar{x}|, 1/2\}|}<\infty.
\end{eqnarray*}
Similarly, one can define $C^{k + \log L}$ for positive integers $k$.

\begin{mthm} \label{mthmw1}  Assume that $f$ is bounded in $Q$. Then there exists a positive constant $C$, such that
$$
\|w\|_{C^{2s, s}_{x, t} (\tilde{Q})} \leq C \|f\|_{L^\infty (Q)}, \;\; \mbox{ if } s \neq 1/2.
$$
$$
\|w\|_{C^{log L, s}_{x, t} (\tilde{Q})} \leq C \|f\|_{L^\infty (Q)}, \;\; \mbox{ if } s = 1/2.
$$
\end{mthm}

\begin{mthm} \label{mthmw2} Assume that $f \in C^{\alpha, \alpha/2}_{x,t}  (Q)$ for some $0 < \alpha < 1$. Then there exists a positive constant $C$, such that
$$
\|w\|_{C^{2s+ \alpha, s+\alpha/2}_{x, t} (\tilde{Q})} \leq C \|f\|_{C^{\alpha, \alpha/2}_{x,t}  (Q)}, \; \; \; \mbox{ if } \; 2s +\alpha < 1 \mbox{ or } 1< 2s +\alpha < 2;
$$
$$
\|w\|_{C^{\log L, s+\alpha/2}_{x, t} (\tilde{Q})} \leq C \|f\|_{C^{\alpha, \alpha/2}_{x,t}  (Q)}, \; \; \; \mbox{ if } \; 2s + \alpha = 1;
$$
$$
\|w\|_{C^{1+\log L, \log L}_{x, t} (\tilde{Q})} \leq C \|f\|_{C^{\alpha, \alpha/2}_{x,t}  (Q)}, \;\;\; \mbox{ if } \; 2s + \alpha = 2.
$$
\end{mthm}

Now Theorems \ref{mthm1}, \ref{mthmv}, \ref{mthmw1}, and \ref{mthmw2} imply Theorems \ref{mthmu1} and \ref{mthmu2}.

\begin{mrem}

It is also clear that $\tilde{Q}$ and $Q$ can be replaced by $Q_r (x^o, t_o)$ and $Q_{2r}(x^o, t_o)$ respectively, where these are the parabolic cylinders centered at any point $(x^o, t_o)$ and of any size $r \geq 1$. As long as the sub-cylinder $\tilde{Q}$ is essentially contained within $Q$-- specifically, if
the distance between $\partial \tilde{Q}$ and $\partial Q$ is bounded below by a positive constant, say $\geq 1$-- the constants $C$ in all estimates remain unchanged.
\end{mrem}
\medskip

\leftline{\bf \large Applications in blow-up and rescaling}
\medskip

Among their numerous applications, the regularity estimates presented in the previous subsection can be utilized to establish a priori estimates for solutions to nonlinear master equations via blow-up and re-scaling analysis.

Consider the equation
\begin{equation}
(\partial_t - \Delta)^s u(x,t) = f(x, u(x,t)), \;\; (x,t) \in \mathbb{R}^n \times \mathbb{R}.
\label{1.20}
\end{equation}

\begin{mthm} \label{mthm7}
Assume that $u$ is a nonnegative solution of (\ref{1.20}) and that the following conditions hold:

(f1) $f(x, \tau)$ is uniformly H\"{o}lder continuous with respect to $x$ and continuous  with respect to $\tau$.

 (f2) There exists a constant $C_0 > 0$ such that
 $$0\leq f(x, \tau) \leq C_0(1+\tau^p) \;\;  \mbox{   uniformly for all $x$, and }$$
 $$
 \mathop{\lim}\limits_{\tau \to \infty}\frac{f(x, \tau)}{\tau^p}=K(x), \,\, 1< p < \frac{n+2}{n+2-2s},
 $$
 where $K(x)\in (0, \infty)$ is uniformly continuous and
 $$ \mathop{\lim}\limits_{|x|\to \infty} K(x)=\bar C\in (0, \infty).$$
 Then there exists a constant $C$, such that  for all nonnegative solutions $u$, we have
 $$
 u(x,t) \leq C \,\,\, \forall \, (x,t) \in \mathbb{R}^n \times \mathbb{R}.
 $$

 \end{mthm}
 The above a priori estimate holds for all $0<s<1$.

 In the case $1/2 <s <1$, we can further establish estimates for more general master equations of the form
 \begin{equation} \label{1.21}
 (\partial_t - \Delta)^s u(x,t) = b(x)|\nabla_x u(x,t)|^q + f(x, u(x,t)), \;\; (x,t) \in \mathbb{R}^n \times \mathbb{R},
 \end{equation}
where $\nabla_x u$ is the gradient of $u$ with respect to $x$.

\begin{mthm} \label{mthm8} Assume that
  $b(x):  \mathbb{R}^n \to \mathbb{R}$  is nonnegative,uniformly bounded and  H\"{o}lder continuous and $f(x, \tau)$ satisfies (f1) and (f2).

If $s>\frac{1}{2}$ and $0<q<\frac{2sp}{2s+p-1},$ then there exists a positive constant $C$ such that
$$
u(x,t)+|\nabla_x u|^{\frac{2s}{2s+p-1}}(x,t)\leq C,  \;\;\; \forall \, (x,t) \in \mathbb{R}^n \times \mathbb{R},
$$
for all nonnegative solutions $u(x,t)$ of \eqref{1.21}.
\end{mthm}
\medskip

For related results concerning a priori estimates, asymptotic behavior, and blow-up rates of solutions for elliptic and parabolic equations, please see \cite{AV, MM, MM1, MV, MZ1, PQ, QS} and the references therein.
\medskip

\leftline{\bf The organization of this paper}

\,\,\,\,\,\,\,\, In Section 2, we provide precise definitions the parabolic H\"{o}lder spaces. We then establish the equivalence between the pseudo differential equation and the integral equation, leading to a proof of Theorem \ref{mthm1}.

In Section 3, we focus on estimating $v(x,t)$.

In Section 4, we derive both H\"{o}lder and Schauder estimates for $w(x,t)$.

Section 5 is devoted to obtaining a priori estimates for nonnegative solutions to master equations \eqref{1.20} and \eqref{1.21}. This is achieved through a detailed analysis involving blow-up and rescaling arguments.

\section{Preliminaries}

\subsection{Parabolic H\"{o}lder spaces}\label{2.1}
We begin by stating the standard definition of parabolic H\"{o}lder space $C^{2\alpha,\alpha}_{x,\, t}(\mathbb{R}^n\times\mathbb{R})$ (cf. \cite{Kry}) as follows.
\begin{itemize}
\item[(i)]
In the case  $0<\alpha\leq\frac{1}{2}$. We say that
$u(x,t)\in C^{2\alpha,\alpha}_{x,\, t}(\mathbb{R}^n\times\mathbb{R})$, if there exists a constant $C>0$ such that
\begin{equation*}
  |u(x,t)-u(y,\tau)|\leq C\left(|x-y|+|t-\tau|^{\frac{1}{2}}\right)^{2\alpha}
\end{equation*}
for all $x,\,y\in\mathbb{R}^n$ and $t,\,\tau\in \mathbb{R}$.
\item[(ii)]
When $\frac{1}{2}<\alpha\leq1$, we say that
$$u(x,t)\in C^{2\alpha,\alpha}_{x,\, t}(\mathbb{R}^n\times\mathbb{R}):=C^{1+(2\alpha-1),\alpha}_{x,\, t}(\mathbb{R}^n\times\mathbb{R}),$$ if $u$ is $\alpha$-H\"{o}lder continuous in $t$ uniformly with respect to $x$ and its gradient $\nabla_xu$ is $(2\alpha-1)$-H\"{o}lder continuous in $x$ uniformly with respect to $t$ and $(\alpha-\frac{1}{2})$-H\"{o}lder continuous in $t$ uniformly with respect to $x$.
\item[(iii)] While for $\alpha>1$, if
$u(x,t)\in C^{2\alpha,\alpha}_{x,\, t}(\mathbb{R}^n\times\mathbb{R}),$
then it means that
$$\partial_tu,\, D^2_xu \in C^{2\alpha-2,\alpha-1}_{x,\, t}(\mathbb{R}^n\times\mathbb{R}).$$
\end{itemize}
In addition, we can analogously define the local parabolic H\"{o}lder space $C^{2\alpha,\alpha}_{x,\, t,\, \rm{loc}}(\mathbb{R}^n\times\mathbb{R})$, which involve local conditions instead of global ones.

\subsection{Integral representation of solutions}

Assume that $u(x,t)$ is a nonnegative entire solution of the master equation
\begin{equation} \label{5.1}
(\partial_t - \Delta)^s u(x,t) = f(x,t), \;\;\; (x,t) \in \mathbb{R}^n \times \mathbb{R}
\end{equation}
with a nonnegative function $f(x,t)$.

We prove Theorem \ref{mthm1}. That is, we show that $u$ can be expressed as
\begin{equation} \label{5.0}
u(x,t) = c + \int_{-\infty}^t \int_{\mathbb{R}^n} f(y,\tau) G(x-y, t-\tau) dy d\tau,
\end{equation}
where
$$G(x-y, t-\tau) = \frac{C_{n,s}}{(t-\tau)^{n/2+1-s}} e^{-\frac{|x-y|^2}{4(t-\tau)}}$$
is the fundamental solution, or the Green's function associated with the master operator $(\partial_t - \lap)^s$, and $c$ is a nonnegative constant.

\begin{proof}[proof of theorem \ref{mthm1}]
First, we need to ensure that the integral
$$\int_{-\infty}^t \int_{\mathbb{R}^n} f(y,\tau) G(x-y, t-\tau) dy d\tau$$
is well-define and convergent.

Let $$Q_R = \{(x,t) \mid |x| \leq R, \, |t| \leq R^2 \}$$
be the parabolic cylinder of size $R$. Let
$$ \chi_R (x,t) = \left\{\begin{array}{ll}
1, & (x,t) \in Q_R, \\
0, & (x,t) \in Q^C_{R}
\end{array}
\right.
$$
be the characteristic function on $Q_R$.

Let $$w_R(x,t) = \int_{-\infty}^t \int_{\mathbb{R}^n} (f \chi_R) (y,\tau)  G(x-y, t-\tau) dy d\tau.$$
Then
$$ (\partial_t - \Delta)^s w_R(x,t) = (f \chi_R) (x,t), \,\, (x, t)\in \mathbb{R}^n\times \mathbb{R},$$
and from the expression of the Green's function, one can  see that
\begin{equation} \label{5.2}
w_R(x,t) \to 0, \; \mbox{ as } |x| \to \infty \; \mbox{ or as } \;  t \to - \infty.
\end{equation}

Set $v_R(x,t) = u(x,t) - w_R(x,t)$. Then
\begin{equation} \label{5.3}
 (\partial_t - \Delta)^s v_R(x,t) = f(x,t)-(f \chi_R) (x,t) \geq 0 , \,\, (x, t)\in \mathbb{R}^n\times \mathbb{R}.
 \end{equation}
We show that
\begin{equation} \label{5.4}
v_R(x,t) \geq 0, \,\, (x, t)\in \mathbb{R}^n\times \mathbb{R}.
\end{equation}
Otherwise, due to the non-negativeness of $u$ and (\ref{5.2}), there exists a negative minimum $(x^o, t_o)$ of $v_R(x,t)$, and it follows that
$$(\partial_t-\Delta)^s v_R(x^o,t_o)
:=C_{n,s}\int_{-\infty}^{t}\int_{\mathbb{R}^n}
  \frac{v_R(x^o,t_o)-v_R(y,\tau)}{(t-\tau)^{\frac{n}{2}+1+s}}e^{-\frac{|x-y|^2}{4(t-\tau)}}\operatorname{d}\!y\operatorname{d}\!\tau <0.
  $$
This contradicts (\ref{5.3}). Therefore, (\ref{5.4}) must be valid. Consequently,
$$ u(x,t) \geq \int_{-\infty}^t \int_{\mathbb{R}^n} (f \chi_R) (y,\tau)  G(x-y, t-\tau) dy d\tau.$$
Now let $R \to \infty$, we arrive at, for each $(x, t)\in \mathbb{R}^n\times \mathbb{R}$,
$$u(x,t) \geq \int_{-\infty}^t \int_{\mathbb{R}^n} f(y,\tau) G(x-y, t-\tau) dy d\tau \equiv \tilde{w}(x,t).$$

Let $$\tilde{v}(x,t) = u(x,t) - \tilde{w}(x,t).$$
Then $\tilde{v}(x,t) \geq 0$, and
$$(\partial_t-\Delta)^s \tilde{v}(x,t) = 0, \,\, (x, t)\in \mathbb{R}^n\times \mathbb{R}.$$

Now (\ref{5.0}) is a consequence of the following Liouville theorem from \cite{CGM}:
\begin{proposition}\label{Liouville}
Let $0<s<1$ and $n\geq2$.
Assume that $\tilde{v}(x,t)$ is a solution of
$$
(\partial_t-\Delta)^s \tilde{v}(x,t) =0, \; \,\,(x, t)\in \mathbb{R}^n\times \mathbb{R}
$$
in the sense of distribution.

In the case $\frac{1}{2}<s < 1$, we assume additionally that
\begin{equation}\label{AA}
  \liminf_{|x|\rightarrow\infty}\frac{\tilde{v}(x,t)}{|x|^\gamma}\geq 0 \; ( \mbox{or} \; \leq 0) \,\,\mbox{for some} \;0\leq\gamma\leq 1.
\end{equation}
Then $\tilde{v}$ must be constant.
\end{proposition}

This completes the proof of Theorem \ref{mthm1}. \end{proof}

\section{Estimates of the homogeneous part v(x,t)}

For convenience in the upcoming long calculations, we choose $Q$ and $\tilde{Q}$ slightly different from their definitions in the Introduction as follows
$$ Q \equiv B_2(0) \times (0,3), \mbox{ and } \; \; \tilde{Q} = B_1(0) \times (1,2) \subset \subset Q.$$
However, as the proofs show, as long as the sub-cylinder $\tilde{Q}$ is essentially contained within $Q$, specifically, if
the distance between $\partial \tilde{Q}$ and $\partial Q$ is bounded below by a positive constant, say $\geq 1$, the constants $C$ in all subsequent estimates remain unchanged.

Set
$$ f_{Q^c} (x,t) = \left\{ \begin{array}{ll} f(x,t), & (x,t) \in Q^c \\
0, & (x, t) \in Q.
\end{array}
\right.
$$

Denote
$$v(x,t) = \int_{-\infty}^t \int_{\mathbb{R}^n} f_{Q^c}(y,\tau) G(x-y, t-\tau) dy d\tau$$
where
$$G(x-y, t-\tau) = \frac{C_{n,s}}{(t-\tau)^{n/2+1-s}} e^{\frac{|x-y|^2}{4(t-\tau)}}$$
is the Green's function associated with the master operator $(\partial_t - \lap)^s$.

Since $v(x,t)$ satisfies the homogeneous equation
$$(\partial_t - \Delta)^s v(x,t) = 0, \;\; (x,t) \in Q,$$
we refer it as the homogeneous part.
\medskip

In this section, we prove Theorem \ref{mthmv}, that is

\begin{theorem} \label{thm2.1} Assume that $v$ is bounded in $Q$, then for any $\alpha \in (0,1)$,
$$
\|v\|_{C^{2, \alpha}_{x, t} (\tilde{Q})} \leq C \|v\|_{L^\infty (Q)}.
$$

\end{theorem}

To prove the theorem, we derive estimates in $x$ and in $t$ separately.

\subsection{Regularity of v in x}

For each fixed $t \in (1,2)$, we obtain

\begin{lemma} \label{thm2.1.1}
There exists a constant $C$ independent of $t\in(1,2),$ such that
$$
\|v(\cdot,t)\|_{C^{2}_{x} (B_1(0))} \leq C \|v(\cdot,t)\|_{L^\infty (B_2(0))}.
$$
\end{lemma}

Here $\|v(\cdot,t)\|_{C^{2}_{x} (B_1(0))}$ is the $C^2$ norm of $v(\cdot, t)$ in $x \in B_1(0)$ for each fixed $t$. Similar notation will be used later.

\begin{proof} Notice that
\begin{eqnarray}\nonumber
    \label{int-v}v(x,t)&=&\int_{-\infty}^t \int_{\mathbb{R}^n}  {f_{Q^c}(y,\tau)}G(x-y,t-\tau)dy d\tau\\
&=&C_{n,s}\int_{-\infty}^t \int_{\mathbb{R}^n}\frac{{f_{Q^c}(y,\tau)}}{(t-\tau)^{n/2+1-s}} e^{-\frac{|x-y|^2}{4(t-\tau)}}dyd\tau.\end{eqnarray}

Evaluating the first and second derivatives of $G(x,t)$  with respect to $x$, we have
\[\frac{\partial}{\partial {x_i}}G(x,t)= \frac{C_{n,s}}{t^{n/2+1-s}}\frac{x_i}{2t}e^{-\frac{|x|^2}{4t}},\]
and
\[\frac{\partial^2}{\partial {x_i}\partial {x_j}}G(x,t)= \frac{C_{n,s}}{t^{n/2+1-s}}\left[\frac{\delta_{ij}}{2t} +\frac{x_ix_j}{4t^2}\right]e^{-\frac{|x|^2}{4t}}.\]
Then from \eqref{int-v}, we obtain\begin{equation}\label{int-vi}
v_{x_i}(x,t)=C_{n,s}\int_{-\infty}^t \int_{\mathbb{R}^n}  \frac{f_{Q^c}(y,\tau)}{(t-\tau)^{n/2+1-s}} \frac{x_i-y_i}{2(t-\tau)}e^{-\frac{|x-y|^2}{4(t-\tau)}} dy d\tau,
\end{equation}
and
\begin{equation}\label{int-vij}
v_{x_{i}x_{j}}(x,t)=C_{n,s}\int_{-\infty}^t \int_{\mathbb{R}^n}  \frac{f_{Q^c}(y,\tau)}{(t-\tau)^{n/2+1-s}} \left[\frac{\delta_{ij}}{2(t-\tau)} +\frac{(x_i-y_i)(x_j-y_j)}{4(t-\tau)^2}\right]e^{-\frac{|x-y|^2}{4(t-\tau)}}dy d\tau.
\end{equation}

Our goal is to control $v_{x_i}(x,t)$ and $v_{x_ix_j}(x,t)$ using $v(x,t)$ itself based on the integral expression \eqref{int-v} for all   $(x,t)\in\tilde Q$ . However, it is obviously impossible  to  compare directly two integrands in integral \eqref{int-v} and \eqref{int-vi}, in other words, there does not exist a universal constant $C$ such that the following inequality
\begin{equation}\label{est-fuc}\frac{|x-y|}{t-\tau}e^{-\frac{|x-y|^2}{4(t-\tau)}}\leq Ce^{-\frac{|x-y|^2}{4(t-\tau)}}, \ \forall(y,\tau)\in Q^c,\end{equation}
holds, because the extra term $\frac{|x-y|}{t-\tau}$ on the left hand side is unbounded.  To overcome this difficulty,  we divide the integration region $Q^c$ into  finite many parts $\{I_j\}_{j\in\{0,1,2,...N\}}$, then, in each part $I_j$, we choose one point $x^j$ in a neighborhood of $x$, and  modify estimation   \eqref{est-fuc} as
\begin{eqnarray}
\frac{|y-x|^2}{(t-\tau)^2} e^{-\frac{|y-x|^2}{4(t-\tau)}}
& \leq & Ce^{-\frac{|y-x^j|^2}{4(t-\tau)}}, \ \forall(y,\tau)\in I_j \label{key est}
\end{eqnarray}
with universal constant $C$.

To this end,
for each fixed  $x\in B_1(0)$, we choose   $I_j $   to be the $j$th quadrant  in the n-dimensional coordinate system centered at $x$, with $ j=1,2,...,2^n.$ And in each $I_j$, we pick a point $$x^j\in \partial B_{\delta}(x)\cap l_j,$$ where $l_j$ denotes the diagonal line of the $j$th quadrant $I_j,$ and $\delta$ be a sufficient small constant to be determined later.

Now for any $y\in B^c_2(0)\cap I_j$,  let us consider the  triangles $\Delta xyx^j$. Let $\theta$ be the angle between $x^j-x$ and $y-x,$ then  it's easy to see  that$$ \cos\theta \geq  \frac{1}{\sqrt{n}}.$$ Hence, it follows from {\em  the law of cosine } that
\begin{eqnarray} \nonumber|y-x|^2&=& |y-x^j|^2-|x^j-x|^2+2|y-x|\cdot|x^j-x|\cdot\cos \theta, \\ \nonumber
&\geq &|y-x^j|^2-\delta^2+\frac{2}{\sqrt{n}}\delta|y-x|\\
&\geq&|y-x^j|^2+\delta^2|y-x|.
\end{eqnarray}
Here we may take $\delta =\frac{1}{\sqrt{n}}$  as  a universal constant,  due to the  fact that $|y-x|\geq1$.

(The following Figure 1 shows the case of $n=2$. Here $\R^2$ is divided into four quadrants centered at $x$.  For $y$ in the third quadrant $I_3$, we choose $x^3$ as the intersection of $\partial B_{1/\sqrt{2}}(x)$ and the diagonal $l_3$.)
\begin{center}
\begin{tikzpicture}
    \def\circleA{2} 
    \def\circleO{2}
    \def\circleB{4} 

    \draw[thick,-, blue] (-5,0) -- (6,0) node[right] {};
    \draw[thick,-,blue] (0,-5) -- (0,5) node[above] {};
    \draw[thick, blue] (0,0) circle (\circleA) node[xshift=-38pt, yshift=12pt] {$B_1(x)$};
    \draw[thick](1,-0.5) circle (\circleO)node[xshift=40pt, yshift=-10pt] {$B_1(o)$};
    \draw[thick] (1, -0.5) circle (\circleB)node[xshift=120pt, yshift=-50pt] {$B_2(o)$} ;

    \filldraw[black] (1, -0.5) circle (2pt) node[below] {$o$};
      \coordinate (x) at (0,0);
      \def\radius{2};
      \path[name path=circle] (x) circle (\radius);
      \path[name path=line1] (-4,-5) -- (4.2,5);
      \path[name path=line2] (4,-5) -- (-4,5);
      \path[name intersections={of=circle and line1, by={P1,P2}}];
      \path[name intersections={of=circle and line2, by={Q1,Q2}}];
    \filldraw[red] (P1) circle (2pt) node[xshift=-4pt, yshift=10pt] {${x}^1$};
    \filldraw[red] (P2) circle (2pt) node[xshift=4pt, yshift=-10pt] {${x}^3$};
    \filldraw[red] (Q1) circle (2pt) node[above] {${x}^2$};
    \filldraw[red] (Q2) circle (2pt) node[xshift=12pt] {${x}^4$};
    \filldraw[red] (-5,-2)  circle (2pt) node[below] {$y$};

    \filldraw (0,0)  circle (2pt) node[right] {$x$};

    \draw[thick, purple] (-3.3,-4) -- (3.5,4) node[above] {};
    \draw[thick, purple] (3.2,-4) -- (-3.2,4) node[above] {};

    \draw[thick, green] (-5,-2) -- (0,0) node[above] {} ;
    \draw[thick, green] (-5,-2) -- (P2) node[below] {} ;

    \node[blue] at (2.9,-4.5) {$I_4$};
    \node[blue] at (2.9, 4.5) {$I_1$};
    \node[blue] at (-2.9,4.5) {$I_2$};
    \node[blue] at (-2.9,-4.5) {$I_3$};
     \node[purple] at (3.5,-4) {$l_4$};
    \node[purple] at (3.5, 3.5) {$l_1$};
    \node[purple] at (-3.2,4.2) {$l_2$};
    \node[purple] at (-3.4,-4.2) {$l_3$};
     \node[purple] at (-0.23,-0.23) {$\theta$};
\node [below=1cm, align=flush center,text width=8cm] at (0,-4.1)
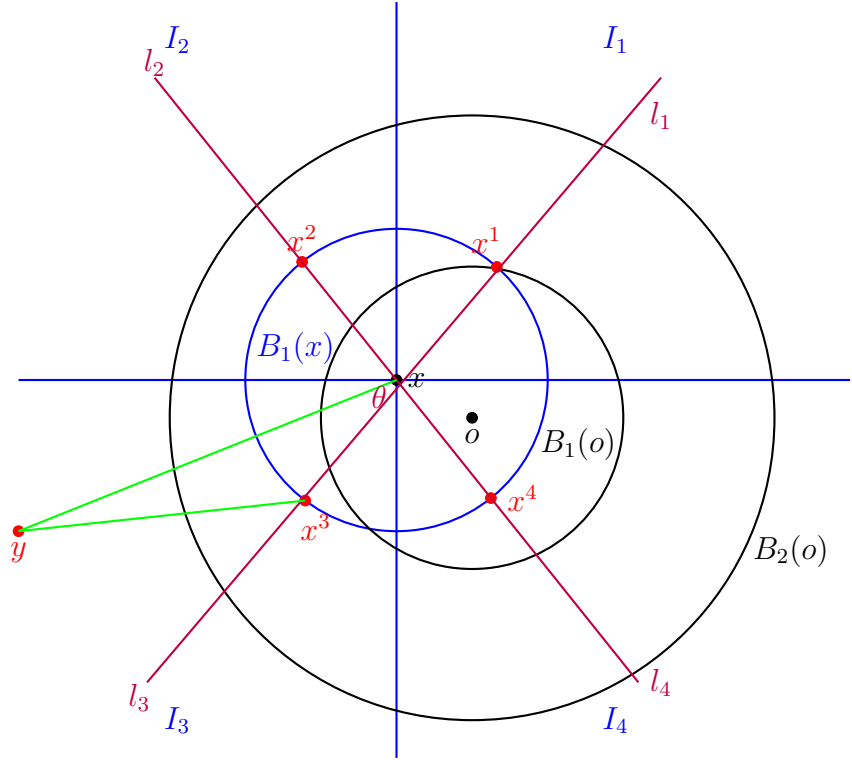
        {$Figure$ 1.  {The case of $n=2$}.};

\end{tikzpicture}
\end{center}

Subsequently, for any fixed $t\in (1,2)$ and $\tau<t,$ we obtain
\begin{eqnarray}\label{est key}
\frac{|y-x|}{t-\tau} e^{-\frac{|y-x|^2}{4(t-\tau)}}
&\leq & \frac{|y-x|}{t-\tau} e^{-\frac{|y-x^j|^2+\delta^2|y-x|}{4(t-\tau)}} \nonumber \\
& \leq & \frac{|y-x|}{t-\tau} e^{-\delta^2\frac{|y-x|}{(t-\tau)}}e^{-\frac{|y-x^j|^2}{4(t-\tau)}}\nonumber\\
& \leq & Ce^{-\frac{|y-x^j|^2}{4(t-\tau)}},
\end{eqnarray}
here, we have used the basic inequality
$$X^p\leq  C_\delta e^{\delta^2X}, \ \forall X\geq 0, p>0.$$
 Furthermore, we can also derive a general type inequality as follows,
\begin{eqnarray}\label{key est2}
\left(\frac{|y-x|}{t-\tau}\right)^p e^{-\frac{|y-x|^2}{4(t-\tau)}}
\leq & Ce^{-\frac{|y-x^j|^2}{4(t-\tau)}}, \  \forall y\in I_j\cap B^c_2(0), \tau<t, \forall p>0.
\end{eqnarray}
The above arguments also validates Lemma \ref{key0}.

Now for $y\in B_2(0), \tau\leq 0,$ since $(t,x)\in \tilde{Q},$ we immediately have
\begin{eqnarray}\label{key est3}
\left(\frac{|y-x|}{t-\tau}\right)^p e^{-\frac{|y-x|^2}{4(t-\tau)}}
\leq & Ce^{-\frac{|y-x|^2}{4(t-\tau)}}, \  \forall y\in B_2(0), \tau<0, \forall p>0.
\end{eqnarray}
Consequently, for fixed $(t,x)\in \tilde{Q},$  write $x^0=x,$  then from \eqref{key est2}, \eqref{key est3} (with $p=1$) and $$supp f_{Q^c}=\bigcup_{j=1}^{2^n} \left([I_j\cap B_2^c(0)]\times(-\infty,t)\right)\cup \left(B_2(0)\times(-\infty,0)\right),$$ we obtain
 \begin{eqnarray}\label{est-vi}
|v_{x_i}(x,t)|&\leq& C_{n,s}\int_{-\infty}^t \int_{\mathbb{R}^n}  \frac{f_{Q^c}(y,\tau)}{(t-\tau)^{n/2+1-s}} \frac{|x-y|}{2(t-\tau)}e^{-\frac{|x-y|^2}{4(t-\tau)}} dy d\tau\nonumber\\
&\leq& C_{n,s}C\left\{\Sigma_{j=1}^{2^n}\int_{-\infty}^t \int_{B_2^c(0)\cap I_j}  \frac{f(y,\tau)}{(t-\tau)^{n/2+1-s}} e^{-\frac{|x^j-y|^2}{4(t-\tau)}} dy d\tau\right.\nonumber\\
&\quad& \quad\quad\quad+\left.\int_{-\infty}^0 \int_{B_2(0)}  \frac{f(y,\tau)}{(t-\tau)^{n/2+1-s}} e^{-\frac{|x-y|^2}{4(t-\tau)}} dy d\tau\right\}\nonumber\\
&\leq& C\Sigma_{j=0}^{2^n}v(x^j,t)\nonumber\\
&\leq& C\|v\|_{L^{\infty}(Q)}.\end{eqnarray}

Now we estimate the second derivatives of $v$. 
\begin{equation}\label{est-vij}
|v_{x_{i}x_{j}}(x,t)|\leq C_{n,s}\int_{-\infty}^t \int_{\mathbb{R}^n}  \frac{f_{Q^c}(y,\tau)}{(t-\tau)^{n/2+1-s}} \left[\frac{1}{2(t-\tau)} +\frac{|x-y|^2}{4(t-\tau)^2}\right]e^{-\frac{|x-y|^2}{4(t-\tau)}}dy d\tau:=I_1+I_2.
\end{equation}

First we estimate $I_1$ by dividing  the integral region of variable $\tau$ into two parts: $t-\tau>1$ and $t-\tau\leq 1.$

If $t-\tau>1$, then \begin{equation}\label{t-small}\frac{1}{2(t-\tau)} e^{-\frac{|x-y|^2}{4(t-\tau)}}\leq C_0 e^{-\frac{|x-y|^2}{4(t-\tau)}}.\end{equation}

If $t-\tau\leq 1$,   from $supp f_{Q^c}\subset Q^c$, there must be $y\in B_2^c(0)$, then
\[\frac{1}{2(t-\tau)}e^{-\frac{|x-y|^2}{4(t-\tau)}}\leq C\frac{|x-y|}{t-\tau}e^{-\frac{|x-y|^2}{4(t-\tau)}}.\]
Fuethermore, if $y\in I_j \cap B_2^c(0)$, by the key estimate \eqref{est key},
we derive that
\begin{equation}\label{t-large}\frac{1}{2(t-\tau)}e^{-\frac{|x-y|^2}{4(t-\tau)}}\leq  Ce^{-\frac{|y-x^j|^2}{4(t-\tau)}}.
\end{equation}
From \eqref{t-small} and \eqref{t-large}, we deduce
 \begin{eqnarray}\label{est-1}
I_1&\leq& C_{n,s}\int_{-\infty}^t \int_{\mathbb{R}^n}  \frac{f_{Q^c}(y,\tau)}{(t-\tau)^{n/2+1-s}} \frac{1}{2(t-\tau)}e^{-\frac{|x-y|^2}{4(t-\tau)}} dy d\tau\nonumber\\
&=&C_{n,s}\left\{\int_{-\infty}^{t-1}+\int_{t-1}^t\right\} \int_{\mathbb{R}^n}  \frac{f_{Q^c}(y,\tau)}{(t-\tau)^{n/2+1-s}} \frac{1}{2(t-\tau)}e^{-\frac{|x-y|^2}{4(t-\tau)}} dy d\tau\nonumber\\
&\leq& C_{n,s}C_0 \int_{-\infty}^{t-1}\int_{\mathbb{R}^n}  \frac{f_{Q^c}(y,\tau)}{(t-\tau)^{n/2+1-s}} e^{-\frac{|x-y|^2}{4(t-\tau)}} dy d\tau\nonumber\\
\ \ &\quad& \ +\Sigma_{j=1}^{2^n}C_{n,s}C\int_{t-1}^t \int_{B_2^c(0)\cap I_j}  \frac{f_{Q^c}(y,\tau)}{(t-\tau)^{n/2+1-s}} e^{-\frac{|x^j-y|^2}{4(t-\tau)}} dy d\tau\nonumber\\
&\leq& C_{n,s}\{C_0 v(x,t)+C\Sigma_{j=1}^{2^n}v(x^j,t)\}\nonumber\\
&\leq& C\|v\|_{L^{\infty}(Q)}.\end{eqnarray}
In addition, it follows from inequality \eqref{key est2} (with $p=2$) that
\begin{eqnarray}\label{est-11}
I_2&\leq& C_{n,s}\int_{-\infty}^t \int_{\mathbb{R}^n}  \frac{f_{Q^c}(y,\tau)}{(t-\tau)^{n/2+1-s}} \frac{|x-y|^2}{4(t-\tau)^2}e^{-\frac{|x-y|^2}{4(t-\tau)}} dy d\tau\nonumber\\
&\leq& C_{n,s}C\Sigma_{j=1}^{2^n}\int_{-\infty}^t \int_{B_2^c(0)\cap I_j}  \frac{f_{Q^c}(y,\tau)}{(t-\tau)^{n/2+1-s}} e^{-\frac{|x^j-y|^2}{4(t-\tau)}} dy d\tau\nonumber\\.
&\leq& C_{n,s}C\Sigma_{j=1}^{2^n}v(x^j,t)\nonumber\\
&\leq& C\|v\|_{L^{\infty}(Q)}.\end{eqnarray}
Combining \eqref{est-vij}, \eqref{est-1} and \eqref{est-11}, we arrive at
\[v_{x_ix_j}(x,t)\leq C\|v\|_{L^{\infty}(Q)}.\]
This completes the proof of Lemma \ref{thm2.1.1}. \end{proof}

\subsection{Regularity of v in t}

For each fixed $x \in B_1(0)$, we derive

\begin{lemma} \label{thm2.1.2}
There exists a constant $C$ independent of $x\in B_1(0),$ such that
$$
\|v(x,\cdot)\|_{C^{\alpha}_{t} ((1,2))} \leq C \|v(x,\cdot)\|_{L^\infty ((0,3))},\ \ \forall \alpha\in(0,1).
$$
\end{lemma}

\begin{proof} Fixed $x\in B_1(0),$ for any $\varphi\in C_0^\infty((1,2)),$
\begin{eqnarray}\label{est-t}
\begin{aligned}
\int_1^2\frac{\partial v(x,t)}{\partial t}\varphi(t)dt&=-\int_1^2v(x,t)\varphi'(t)dt\\
&=-\int_1^2v(x,t)\left(\lim\limits_{\delta\to 0}\frac{\varphi(t)-\varphi(t-\delta)}{\delta}\right)dt\\
&=-\lim\limits_{\delta\to 0}\frac{1}{\delta}\left\{\int_1^2v(x,t)\varphi(t)dt-\int_1^2v(x,t)\varphi(t-\delta)dt\right\}\\
&=-\lim\limits_{\delta\to 0}\frac{1}{\delta}\left\{\int_1^2v(x,t)\varphi(t)dt-\int_{1-\delta}^{2-\delta}v(x,t+\delta)\varphi(t)dt\right\}\\
&=-\lim\limits_{\delta\to 0}\frac{1}{\delta}\left\{\int_1^2[v(x,t)-v(x,t+\delta)]\varphi(t)dt+\left(\int_{2-\delta}^{2}
-\int_{1-\delta}^{1}\right)v(x,t+\delta)\varphi(t)dt\right\}\\
&=-\lim\limits_{\delta\to 0}\int_1^2\frac{v(x,t)-v(x,t+\delta)}{\delta}\varphi(t)dt.
\end{aligned}
\end{eqnarray}
Here we have utilized the fact that
\begin{equation}\label{small-term}\lim\limits_{\delta\to 0}\frac{1}{\delta}\int_{2-\delta}^{2}v(x,t+\delta)\varphi(t)dt=\lim\limits_{\delta\to 0}\frac{1}{\delta}\int_{1-\delta}^{1}v(x,t+\delta)\varphi(t)dt=0.\end{equation}

Indeed, by the {\em mean value theorem for integrals},   the boundedness of $v(x,t)$ and the fact that $\varphi\in C_0^\infty((1,2)),$ we have
\[\lim\limits_{\delta\to 0}\frac{1}{\delta}\int_{2-\delta}^{2}v(x,t+\delta)\varphi(t)dt=\lim\limits_{\xi\to 2}v(x,\xi)\varphi(\xi)=0.\]
Similarly,
\[\lim\limits_{\delta\to 0}\frac{1}{\delta}\int_{1-\delta}^{1}v(x,t+\delta)\varphi(t)dt=\lim\limits_{\xi\to 1}v(x,\xi)\varphi(\xi)=0.\]
These imply \eqref{small-term}.

Now for any fixed $(x,t)\in B_1(0)\times(1,2):=\tilde{Q}$, substituting the expression
\begin{equation}v(x,t)=\int_{-\infty}^t \int_{\mathbb{R}^n} (f_{Q^c}(y,\tau) G(x-y, t-\tau) dy d\tau\end{equation}
into the right hand side  of \eqref{est-t}, we derive
\begin{eqnarray}\label{est-t1}
\begin{aligned}
&\int_1^2\frac{v(x,t+\delta)-v(x,t)}{\delta}\varphi(t)dt\\
=&\frac{1}{\delta}\int_1^2\left\{\int_{-\infty}^{t+\delta}\int_{\R^n}f_{Q^c}(y,\tau)G(x-y,t+\delta-\tau)dyd\tau
-\int_{-\infty}^{t}\int_{\R^n}f_{Q^c}(y,\tau)G(x-y,t-\tau)dyd\tau\right\}\varphi(t)dt\\
=&\int_1^2\int_{-\infty}^{t}\int_{\R^n}f_{Q^c}(y,\tau)\frac{G(x-y,t+\delta-\tau)-G(x-y,t-\tau)}{\delta}dyd\tau\varphi(t)dt\\
&+\frac{1}{\delta}\int_1^2\int_{t}^{t+\delta}\int_{\R^n}f_{Q^c}(y,\tau)G(x-y,t+\delta-\tau)dyd\tau\varphi(t)dt\\
:=&J_1+J_2.\end{aligned}
\end{eqnarray}

First, we want show that $J_1$  is equivalent to the differentiation of $v(x,t)$ with respect to $t$ within the integral sign, that is,
\begin{eqnarray}\label{J1-lim}
\begin{aligned}
J_1\to&\int_1^2\int_{-\infty}^{t}\int_{\R^n}f_{Q^c}(y,\tau)\frac{\partial G}{\partial t} (x-y,t-\tau)dyd\tau\varphi(t)dt, \ \mb{as}\ \ \delta\to 0.
\end{aligned}
\end{eqnarray}

 In fact, applying the Cauchy mean value theorem, we obtain
\begin{eqnarray*}
\begin{aligned}J_1=\int_1^2\int_{-\infty}^{t}\int_{\R^n}f_{Q^c}(y,\tau)\frac{\partial G}{\partial t}(x-y,t+\delta'-\tau)dyd\tau\varphi(t)dt, \delta'\in(0,\delta).\end{aligned}
\end{eqnarray*}
Differentiate $G(x,t)$ with respect to $t$,
\begin{equation}\label{grn-t}\frac{\partial G(x,t)}{\partial t}
=\frac{C_{n,s}}{{t}^{n/2+1-s}} e^{-\frac{|x|^2}{4t}}(\frac{|x|^2}{4t^2}-\frac{\frac{n}{2}+1-s}{t}).\end{equation}
Then for $\tau<t, 0<\delta'<\delta,$
\begin{eqnarray}\label{delt}
\begin{aligned}\left|\frac{\partial G}{\partial t}(x-y,t+\delta'-\tau)\right|\leq \frac{C e^{-\frac{|x-y|^2}{4(t+\delta'-\tau)}}}{{(t+\delta'-\tau)}^{n/2+1-s}}\left(\frac{|x-y|^2}{(t+\delta'-\tau)^2}+\frac{1}{t+\delta'-\tau}\right)\end{aligned}
\end{eqnarray}

Similar to the argument as in deriving the estimate of $v_{x_i}$, we estimate the right hand side of  \eqref{delt} in two possible cases.

Case 1. $y\in B_2^c(0),$ then $|y-x|\geq 1.$ We can choose $x^j$ as above such that for each $j=1,2,...2^n,$ if $y\in I_j,$ by \eqref{key est2}(with $p=1,2$), it holds
\begin{eqnarray}\label{delt1}
\begin{aligned}e^{-\frac{|x-y|^2}{4(t+\delta'-\tau)}}\left(\frac{|x-y|^2}{(t+\delta'-\tau)^2}+\frac{1}{t+\delta'-\tau}\right)\leq Ce^{-\frac{|x^j-y|^2}{4(t+\delta'-\tau)}}.
\end{aligned}
\end{eqnarray}

Case 2. $y\in B_2(0), \tau\leq 0,$ then
\[|x-y|\leq 3, t+\delta'-\tau\geq1.\]
Therefore,
\begin{eqnarray}\label{delt2}
\begin{aligned}e^{-\frac{|x-y|^2}{4(t+\delta'-\tau)}}\left(\frac{|x-y|^2}{(t+\delta'-\tau)^2}+\frac{1}{t+\delta'-\tau}\right)\leq Ce^{-\frac{|x-y|^2}{4(t+\delta'-\tau)}}.
\end{aligned}
\end{eqnarray}

Substituting \eqref{delt1} and \eqref{delt2} into \eqref{delt}, we deduce
\begin{eqnarray}\label{delt3}
\begin{aligned}\left|\frac{\partial G}{\partial t}(x-y,t+\delta'-\tau)\right|&\leq C\sum\limits_{j=0}^{2^n}\frac{e^{-\frac{|x^j-y|^2}{4(t+\delta'-\tau)}}}{{(t+\delta'-\tau)}^{n/2+1-s}}\\
&\leq  C\sum\limits_{j=0}^{2^n} G(x^j-y,t+\delta'-\tau)\\
&\leq  C\sum\limits_{j=0}^{2^n}\{ G(x^j-y,t-\tau)+ G(x^j-y,t+1-\tau)\},\end{aligned}
\end{eqnarray}
{with} $x^0=x$.
In addition,
\begin{eqnarray}\label{delt4}
\begin{aligned}&\int_1^2\int_{-\infty}^{t}\int_{\R^n}f_{Q^c}(y,\tau)\sum\limits_{j=0}^{2^n}\{ G(x^j-y,t-\tau)+ G(x^j-y,t+1-\tau)\}dyd\tau\varphi(t)dt\\
&\leq C\left(\sum\limits_{j=0}^{2^n}\int_1^2{v(x^j,t+1)\varphi(t)dt+\int_1^2v(x^j,t)}\varphi(t)dt\right)\\
&\leq C \|v\|_{L^{\infty}(Q)}\leq C.\end{aligned}
\end{eqnarray}

Therefore, taking into account of \eqref{delt3}, \eqref{delt4}, and by the {\em Lebesgue's dominated convergence theorem}, we verify \eqref{J1-lim}.  This implies that, as $\delta\to 0,$
\begin{eqnarray}\label{est-J1}
\begin{aligned}
J_1\to&\int_1^2\int_{-\infty}^{t}\int_{\R^n}\frac{C_{n,s}f_{Q^c}(y,\tau)}{{(t-\tau)}^{n/2+1-s}} e^{-\frac{|x-y|^2}{4(t-\tau)}}\left[\frac{|x-y|^2}{4(t-\tau)^2}-\frac{\frac{n}{2}+1-s}{t-\tau}\right]dyd\tau\varphi(t)dt.\\
\end{aligned}
\end{eqnarray}

Next we claim \begin{equation}\label{est-J2}J_2\to 0, \ \mb{as}\ \ \delta\to 0.\end{equation}

Indeed, for any $0<\delta<1,$ if $\tau\in(0,\delta)$, then $\tau+t\in(0,3)$, by $supp f_{Q^c}\subset Q^c$,
\[D:=\{y\in\R^n:f_{Q^c}(y,t+\tau)\neq 0\}\subset B_2^c(0).\]
 Now for $\forall y\in D$, $|y-x|\geq 1$,  by equality \eqref{grn-t}, we observe that $G(|x-y|,\delta)$ is monotone increasing in $\delta$ for $\delta\ll1$.  
Thus, for sufficiently small $\delta>0,$ we obtain
\begin{eqnarray}\label{est-t2}
\begin{aligned}
J_2&=\frac{1}{\delta}\int_1^2\int_{t}^{t+\delta}\int_{\R^n}f_{Q^c}(y,\tau)G(x-y,t+\delta-\tau)dyd\tau\varphi(t)dt\\
&=\frac{1}{\delta}\int_1^2\int_{0}^{\delta}\int_{\R^n}f_{Q^c}(y,\tau+t)G(x-y,\delta-\tau)dyd\tau\varphi(t)dt\\
&\leq\frac{1}{\delta}\int_1^2\int_{0}^{\delta}\int_{\R^n}f_{Q^c}(y,\tau+t)G(x-y,\delta)dyd\tau\varphi(t)dt\\
&=\frac{1}{\delta}\int_{0}^{\delta}\int_{\R^n}\int_1^2f_{Q^c}(y,\tau+t)G(x-y,\delta)\varphi(t)dtdyd\tau\\
&=\frac{1}{\delta}\int_{0}^{\delta}\int_{\R^n}\int_{1+\tau}^{2+\tau}f_{Q^c}(y,t)G(x-y,\delta)\varphi(t-\tau)dtdyd\tau\\
&\leq \sup_{t\in(1,2)}\varphi(t)\frac{1}{\delta}\int_{0}^{\delta}\int_{\R^n}\int_{1}^{\frac{5}{2}}f_{Q^c}(y,t)G(x-y,\delta)dtdyd\tau\\
&\leq C\int_{\R^n}\int_{1}^{\frac{5}{2}}f_{Q^c}(y,t)G(x-y,\delta)dtdy:=J_3.\\
\end{aligned}
\end{eqnarray}
Therefore, it suffices to prove
\begin{equation}\label{est-J3}J_3=C\int_{\R^n}\int_{1}^{\frac{5}{2}}f_{Q^c}(y,\tau)G(x-y,\delta)d\tau dy\to 0, \ \mb{as}\ \ \delta\to 0.\end{equation}

In fact,
\[f_{Q^c}(y,\tau)G(x-y,\delta)=f_{Q^c}(y,\tau)\frac{e^{-\frac{|x-y|^2}{4\delta}}}{{\delta}^{n/2+1-s}}\to 0 \ a.e.\ \mb{as}\ \ \delta\to 0. \]
 Furthermore, there exists a point $t_0\in (\frac{5}{2},3)$ such that for any $\tau\in(1,\frac{5}{2})$ and sufficiently small $\delta>0$, it holds
 \[f_{Q^c}(y,\tau)G(x-y,\delta)\leq f_{Q^c}(y,\tau)G(x-y,t_0-\tau)\]
 and \begin{eqnarray*}
\begin{aligned}
&\int_{\R^n}\int_1^3f_{Q^c}(y,\tau)G(x-y,t_0-\tau)d\tau dy\\
\leq &\int_{\R^n}\int_{-\infty}^{t_0}f_{Q^c}(y,\tau)G(x-y,t_0-\tau)d\tau dy\\
=&v(x,t_0)\leq \|v\|_{L^{\infty}(Q)}.\end{aligned}
\end{eqnarray*}
Hence, by the {\em Lebesgue's dominated convergence theorem}, \eqref{est-J3} is valid.

Together \eqref{est-t2} with \eqref{est-J3}, we verified claim \eqref{est-J2}.

Then combining \eqref{est-t}, \eqref{est-t1},\eqref{est-J1} and \eqref{est-J2}, we arrive at
\[\int_1^2\frac{\partial v(x,t)}{\partial t}\varphi(t)dt=C_{n,s}\int_1^2\int_{-\infty}^{t}\int_{\R^n}\frac{f_{Q^c}(y,\tau)}{{(t-\tau)}^{n/2+1-s}} e^{-\frac{|x-y|^2}{4(t-\tau)}}\left[\frac{|x-y|^2}{4(t-\tau)^2}-\frac{\frac{n}{2}+1-s}{t-\tau}\right]dyd\tau\varphi(t)dt,\]
which implies that
\[\frac{\partial v(x,t)}{\partial t}=C_{n,s}\int_{-\infty}^{t}\int_{\R^n}\frac{f_{Q^c}(y,\tau)}{{(t-\tau)}^{n/2+1-s}} e^{-\frac{|x-y|^2}{4(t-\tau)}}\left[\frac{|x-y|^2}{4(t-\tau)^2}-\frac{\frac{n}{2}+1-s}{t-\tau}\right]dyd\tau, a.e. (x,t)\in \tilde{Q}.\]
Now similar to the argument for the estimate of $v_{x_ix_j}$ in \eqref{est-vij}, we deduce that the weak derivative of $v(x,t)$ with respect to the variable $t$ is controlled by $\|v\|_{L^{\infty}(Q)}$, that is,
\[\Bigg|\frac{\partial v(x,t)}{\partial t}\Bigg|\leq C\|v\|_{L^{\infty}(Q)}.\]
Therefore, for any fixed $x\in B_1(0),$ $v(x,\cdot)\in W^{1,\infty}((1,2)),$  by the sobolev  embedding theorem, we deduce that $v(x,\cdot)\in C_t^{0,1}((1,2))\subset C_t^{\alpha}((1,2)),  \, \forall \alpha\in(0, 1),$ and
\[\|v(x,\cdot)\|_{C_t^{\alpha}((1,2))}\leq C \|v(x,\cdot)\|_{L^{\infty}((0,3))}.\]
This completes the proof of Lemma \ref{thm2.1.2}. \end{proof}
\medskip

Now Theorem \ref{thm2.1} is a direct consequence of Lemma \ref{thm2.1.1} and Lemma \ref{thm2.1.2}.

\section{Estimates of the nonhomogeneous part w(x,t)}

As in the previous section, we still denote
$$ Q \equiv B_2(0) \times (0,3), \mbox{ and } \; \; \tilde{Q} = B_1(0) \times (1,2) \subset \subset Q.$$

Let
$$ f_Q (x,t) = \left\{ \begin{array}{ll} f(x,t), & (x,t) \in Q \\
0, & (x, t) \in Q^c,
\end{array}
\right.
$$

and
$$w(x,t) = \int_{-\infty}^t \int_{\mathbb{R}^n} f_Q(y,\tau) G(x-y, t-\tau) dy d\tau$$
where
$$G(x-y, t-\tau) = \frac{C_{n,s}}{(t-\tau)^{n/2+1-s}} e^{\frac{|x-y|^2}{4(t-\tau)}}$$

Since $w(x,t)$ satisfies the nonhomogeneous equation
$$(\partial_t - \Delta)^s w(x,t) = f(x,t),$$
we refer it as the nonhomogeneous part.
\smallskip

In this section, we estimate H\"{o}lder and Schauder norms of $w(x,t)$ on $\tilde{Q}$ in terms of $L^\infty$ and H\"{o}lder norms of $f(x,t)$ on the larger cylinder $Q$. We prove

\begin{theorem} \label{thm3.1}  (H\"{o}lder Estimates) Assume that $f$ is bounded in $Q$. Then there exists a positive constant $C$, such that
$$
\|w\|_{C^{2s, s}_{x, t} (\tilde{Q})} \leq C \|f\|_{L^\infty (Q)}, \;\; \mbox{ if } s \neq 1/2.
$$
$$
\|w\|_{C^{log L, s}_{x, t} (\tilde{Q})} \leq C \|f\|_{L^\infty (Q)}, \;\; \mbox{ if } s = 1/2.
$$
\end{theorem}

\begin{theorem} \label{thm3.2} (Schauder Estimates) Assume that $f \in C^{\alpha, \alpha/2}_{x,t}  (Q)$ for some $0 < \alpha < 1$. Then there exists a positive constant $C$, such that
$$
\|w\|_{C^{2s+ \alpha, s+\alpha/2}_{x, t} (\tilde{Q})} \leq C \|f\|_{C^{\alpha, \alpha/2}_{x,t}  (Q)}, \; \; \; \mbox{ if } \; 2s +\alpha < 1 \mbox{ or } 1< 2s +\alpha < 2;
$$
$$
\|w\|_{C^{\log L, s+\alpha/2}_{x, t} (\tilde{Q})} \leq C \|f\|_{C^{\alpha, \alpha/2}_{x,t}  (Q)}, \; \; \; \mbox{ if } \; 2s + \alpha = 1;
$$
$$
\|w\|_{C^{1+\log L, \log L}_{x, t} (\tilde{Q})} \leq C \|f\|_{C^{\alpha, \alpha/2}_{x,t}  (Q)}, \;\;\; \mbox{ if } \; 2s + \alpha = 2.
$$
\end{theorem}

\subsection{The bounded-ness of w}

We first show
\begin{lemma} \label{lem4.1}
Assume that $f(x,t)$ is bounded in $Q$. Then  there exists a constant $C$, independent of $x \in B_1(0)$ and $t \in (1,2)$, such that
\begin{equation}
|w(x,t)| \leq C \|f\|_{L^\infty (Q)}.
\label{Ewx0}
\end{equation}
\end{lemma}

\begin{proof} In fact
\begin{eqnarray*}
|w(x,t)| &\leq &  \int_{-\infty}^t \int_{\mathbb{R}^n} | f_Q(y,\tau)| G(x-y, t-\tau) d y d\tau \nonumber \\
& \leq & C \, \|f\|_{L^\infty (Q)} \int_{0}^t \int_{B_2(0)} \frac{C_{n,s}}{(t-\tau)^{n/2+1-s}} e^{-\frac{|x-y|^2}{4(t-\tau)}} d y d \tau \nonumber \\
& \leq & C_1 \, \|f\|_{L^\infty (Q)} \int_{0}^t \frac{1}{(t-\tau)^{1-s}} \int_{\mathbb{R}^n} e^{-|z|^2} d \, z \, d \, \tau \nonumber \\
& \leq & C_2 \, \|f\|_{L^\infty (Q)} \int_{0}^t \frac{1}{(t-\tau)^{1-s}} d \, \tau \nonumber \\
& \leq & C_3 \, \|f\|_{L^\infty (Q)}.
\end{eqnarray*}
This completes the proof of Lemma \ref{lem4.1}. \end{proof}

\subsection{Regularity of w in x}

\subsubsection{H\"{o}lder estimate in the case $s \leq 1/2$}

We prove
\begin{lemma} \label{lem4.2}
Assume that $f(x,t)$ is bounded in $Q$.
Then there exists a constant $C>0$ independent of $t \in (1,2)$ and $x, \bar{x} \in B_1(0)$, such that
\begin{equation}
|w(x,t) - w(\bar{x}, t) | \leq \left\{\begin{array}{ll} C \|f\|_{L^\infty (Q)} |x-\bar{x}|^{2s}, \mbox{ if } s < 1/2,\\
C \|f\|_{L^\infty (Q)} |x-\bar{x}|  |\log \min\{|x-\bar{x}|, 1/2\}|, \mbox{ if } s = 1/2.
\end{array}
\right.
\label{Ewx1}
\end{equation}
\end{lemma}

\begin{proof}

It suffice to derive that
\begin{equation}
\int_{0}^t  \int_{B_2(0)} | G(x-y, t-\tau) - G(\bar{x}-y, t-\tau) | d y d \tau  \leq C |x-\bar{x}|^{2s}.
\label{3.1}
\end{equation}

Denote $\delta = |x-\bar{x}|$, and let $\eta$ be the midpoint on the line segment between $x$ and $\bar{x}$. Since $x, \bar{x} \in B_2(0)$, which is contained in $B_3(\eta)$, for convenience, we estimate the integrals in two parts of $B_3(\eta)$:  $B_\delta (\eta)$ and $B_3(\eta) \setminus B_\delta (\eta)$ respectively.

For $B_\delta (\eta)$ part, we calculate each individual integral separately.

By making change of variables $ \hat{\tau} = |x-y|^2/4(t-\tau)$,
the first integral
$$\int_0^t \int_{B_\delta (\eta)} \frac{C_{n,s}}{(t-\tau)^{n/2+1-s}} e^{-\frac{|x-y|^2}{4(t-\tau)}} dy d\tau$$
becomes
$$ \int_{B_\delta (\eta)} \frac{C}{|x-y|^{n-2s}} \int_{|x-y|^2/t}^\infty \hat{\tau}^{n/2-1-s} e^{-\hat{\tau}} d \hat{\tau} d y.$$
Under the assumption that $n \geq 2$, the integral
$$ \int_{|x-y|^2/t}^\infty \hat{\tau}^{n/2-1-s} e^{-\hat{\tau}} d \hat{\tau} $$
is bounded. Hence
\begin{equation}
\int_0^t  \int_{B_\delta (\eta)} | G(x-y, t-\tau)| d y d \tau \leq C \int_{B_\delta (\eta)} \frac{1}{|x-y|^{n-2s}} d y \leq C_1 \, \delta^{2s}.
\label{3.5}
\end{equation}

Similarly
\begin{equation}
\int_0^t  \int_{B_\delta (\eta)} | G(\bar{x}-y, t-\tau)| d y d \tau \leq C \int_{B_\delta (\eta)} \frac{1}{|\bar{x}-y|^{n-2s}} d y \leq C_1 \, \delta^{2s}.
\label{3.6}
\end{equation}

To estimate the integral of the difference on $B_3(\eta) \setminus B_\delta (\eta)$, we apply the Mean Value Theorem and let $\xi$ be some point between $x$ and $\bar{x}$ to arrive at
\begin{eqnarray}
& &\int_0^t  \int_{B_3(\eta) \setminus B_\delta (\eta)} | G(x-y, t-\tau) - G(\bar{x}-y, t-\tau) | d y d \tau \nonumber \\
& \leq & C \int_0^t \frac{1}{(t-\tau)^{n/2+1-s}} \int_{B_3(\eta) \setminus B_\delta (\eta)}  e^{-\frac{|\xi-y|^2}{4(t-\tau)}} \frac{|\xi-y|}{t-\tau} |x - \bar{x}| \, d y \, d \tau \nonumber \\
& = & C |x-\bar{x}| \int_{B_3(\eta) \setminus B_\delta (\eta)} \frac{1}{|\xi -y|^{n+1-2s}} \int_{\frac{|\xi -y|^2}{t}}^\infty \tau^{n/2-s} e^{-\tau} d \tau d y \nonumber \\
& \leq & C_1 \, \delta \, ( C_2 + \delta^{2s-1}) \nonumber \\
& \leq & C_3 \, \delta^{2s}, \;\; \; \mbox{ if } s < 1/2.
\label{3.7}
\end{eqnarray}
In the case $s = 1/2$, the last two lines of the above estimate becomes
$$ \leq C_1 \delta \, ( C_2 + |\log \delta |) \leq C_3 \delta \, |\log \delta|.$$

Here we have used the fact that in $B_3(\eta) \setminus B_\delta(\eta)$, $|\xi -y|$ is equivalent to $|\eta -y|$, i.e.
there exist positive constants $c$ and $C$ such that
$$ c |\eta - y| \leq |\xi -y| \leq C |\eta -y| $$
for any point $\xi$ lying on the line segment between $x$ and $\bar{x}$. We will use this fact several time in the following arguments.

Now from (\ref{3.5}), (\ref{3.6}), and (\ref{3.7}) we arrive at (\ref{3.1}) and hence obtain (\ref{Ewx1}). This completes the proof of Lemma \ref{lem4.2} \end{proof}

\subsubsection{H\"{o}lder estimate in the case $1/2 < s < 1$ }

In this case, we derive
\begin{lemma} \label{lem4.3}
 that, there is a constant $C>0$ independent of $t \in (1,2)$ and $x, \bar{x} \in B_1(0)$,
\begin{equation}
\left| \frac{\partial w}{\partial x_i} (x,t) - \frac{\partial w}{\partial x_i} (\bar{x},t) \right| \leq C \|f\|_{L^\infty (Q)} |x-\bar{x}|^{2s-1}.
\label{Ewx2}
\end{equation}
\end{lemma}

\begin{proof} Again it suffice to show
\begin{equation}
\int_0^t \int_{B_2(0)} | \frac{\partial}{\partial x_i} G(x-y, t-\tau) - \frac{\partial}{\partial x_i} G(\bar{x}-y, t-\tau) | d y d \tau  \leq C |x-\bar{x}|^{2s-1}.
\label{3.3}
\end{equation}

Again, set $\delta = |x-\bar{x}|$, and let $\eta$ be the midpoint on the line segment between $x$ and $\bar{x}$. We estimate the integrals in two parts of $B_3(\eta)$: $B_\delta (\eta)$ and $B_3(\eta) \setminus B_\delta (\eta)$ respectively.

For $B_\delta (\eta)$ part, we evaluate each individual integral separately.

\begin{eqnarray}
& & \int_0^t \int_{B_\delta (\eta)} \left| \frac{\partial}{\partial x_i} G(x-y, t-\tau) \right| \, d \, y \, d\, \tau \nonumber \\
& \leq & C  \int_{B_\delta (\eta)} \int_0^t e^{\frac{|x-y|^2}{4(t-\tau)}} \frac{|x-y|}{(t-\tau)^{n/2+2-s}} \, d \, \tau \, d\, y \nonumber \\
& = & C \int_{B_\delta (\eta)} \frac{1}{|x-y|^{n-2s +1}} \int_{\frac{|x-y|^2}{t}} ^\infty \tau^{n/2-s} e^{-\tau}  \, d \, \tau \, d\, y \nonumber \\
& \leq & C_1 \, \delta^{2s-1}.
\label{3.9}
\end{eqnarray}

Similarly,
\begin{equation}   \int_0^t \int_{B_\delta (\eta)} \left| \frac{\partial}{\partial x_i} G(\bar{x}-y, t-\tau) \right| \, d \, y \, d\, \tau \leq C \, \delta^{2s-1}.
\label{3.10}
\end{equation}

To estimate the part on $B_3(\eta) \setminus B_\delta (\eta)$, we apply the Mean Value Theorem to the difference in the integral on the left hand side of (\ref{3.3}).

Denote
$$ g_i(x) =  e^{- \frac{|x-y|^2}{4(t-\tau)}} (x_i -y_i).$$
Then
\begin{eqnarray}
& & \int_0^t \int_{B_3(\eta) \setminus B_\delta (\eta)} | \frac{\partial}{\partial x_i} G(x-y, t-\tau) - \frac{\partial}{\partial x_i} G(\bar{x}-y, t-\tau) | d y d \tau  \nonumber \\
& \leq & \int_0^t \frac{1}{(t-\tau)^{n/2+2-s}} \int_{B_3(\eta) \setminus B_\delta (\eta)} |\nabla g_i(\xi)| |x - \bar{x}| \, d y \, d \tau \nonumber \\
& \leq & C \int_0^t \frac{1}{(t-\tau)^{n/2+2-s}} \int_{B_3(\eta) \setminus B_\delta (\eta)}  e^{-\frac{|\xi-y|^2}{4(t-\tau)}} \left( \frac{|\xi-y|^2}{t-\tau} + 1 \right) |x - \bar{x}| \, d y \, d \tau \nonumber \\
& \leq & C_1 \, \delta \int_{B_3(\eta) \setminus B_\delta (\eta)} \frac{1}{|\xi-y|^{n-2s}} \int_{\frac{|\xi-y|}{t}}^\infty ( \tau^{n/2-s} + \tau^{n/2-1-s} ) e^{-\tau} d \tau d y \nonumber \\
& \leq & C_2 \, \delta \int_{B_3(\eta) \setminus B_\delta (\eta)} \frac{1}{|\xi-y|^{n-2s}} d y \nonumber \\
& \leq & C_3 \, \delta ( 1 + \delta^{2s} ) \nonumber \\
& \leq & C_4 \, \delta.
\label{3.8}
\end{eqnarray}

Now (\ref{3.9}), (\ref{3.10}), and (\ref{3.8}) imply (\ref{3.3}) and hence (\ref{Ewx2}). This completes the proof of Lemma \ref{lem4.3}.
\end{proof}

\subsubsection{Schauder estimate in the case $2s + \alpha \leq 1$}

Assume that for all $t \in (1,2)$, $f(t,x)$ is uniformly $\alpha$-H\"{o}lder continuous in $x$ on $B_2(0)$, and we simply denote this norm by $\|f\|_{C_x^\alpha(Q)}$.

We employ the H\"{o}lder continuity of $f$ to lift the regularity of $w$ and deduce
\begin{lemma} \label{lem4.4}  There exists constant $C>0$ independent of $t \in (1,2)$ and $x, \bar{x} \in B_1(0)$, such that

\begin{equation} \label{Ewx3}
|w(x,t) - w(\bar{x}, t)| \leq  \left\{\begin{array}{ll} C \, \|f\|_{C_x^\alpha(Q)} |x - \bar{x}|^{2s+\alpha}, & \mbox{ if } 2s +\alpha < 1, \\
C \, \|f\|_{C_x^\alpha(Q)} |x - \bar{x}| |\log \min\{|x-\bar{x}|, 1/2\}|, & \mbox{ if } 2s + \alpha = 1.
\end{array}
\right.
\end{equation}
\end{lemma}

\begin{proof} Again, set $\delta = |x-\bar{x}|$, and let $\eta$ be the midpoint on the line segment between $x$ and $\bar{x}$.
By symmetry, we have
\begin{equation} \label{3.10a}
\int_{B_3 (\eta)} e^{- \frac{|x-y|^2}{4(t-\tau)}} d \, y = \int_{B_3 (\eta)} e^{- \frac{|\bar{x}-y|^2}{4(t-\tau)}} d \, y.
\end{equation}
Therefore
\begin{eqnarray}
& & |w(x,t) - w(\bar{x}, t)| \nonumber \\
& = & \left| \int_0^t \frac{1}{(t-\tau)^{n/2+1-s}} \int_{B_3(\eta)} ( f_Q (y,\tau) - f_Q(x,\tau) ) \left(  e^{- \frac{|x-y|^2}{4(t-\tau)}} -
e^{- \frac{|\bar{x}-y|^2}{4(t-\tau)}} \right) d \, y \, d \, \tau \right| \nonumber \\
& \leq & C \, \|f\|_{C_x^\alpha(Q)}   \int_{B_3(\eta)} |x-y|^\alpha \int_0^t \frac{1}{(t-\tau)^{n/2+1-s}} \left|  e^{- \frac{|x-y|^2}{4(t-\tau)}} -
e^{- \frac{|\bar{x}-y|^2}{4(t-\tau)}} \right| d \, y \, d \, \tau
\label{3.14}
\end{eqnarray}

 We estimate the integrals in two parts, $B_3(\eta) \setminus B_\delta (\eta)$ and $B_\delta (\eta)$,  respectively.

On the first part, we apply the Mean Value Theorem to derive
\begin{eqnarray}
& &  \int_{B_3(\eta)\setminus B_\delta (\eta)} |x-y|^\alpha \int_0^t \frac{1}{(t-\tau)^{n/2+1-s}} \left|  e^{- \frac{|x-y|^2}{4(t-\tau)}} -
e^{- \frac{|\bar{x}-y|^2}{4(t-\tau)}} \right| d \, y \, d \, \tau \nonumber \\
& \leq & C  \int_{B_3(\eta) \setminus B_\delta (\eta)} |x-y|^\alpha \int_0^t \frac{1}{(t-\tau)^{n/2+1-s}} e^{-\frac{|\xi-y|^2}{4(t-\tau)}} \frac{|\xi-y|}{t-\tau} |x - \bar{x}| \, d y \, d \tau \nonumber \\
& \leq & C \, |x-\bar{x}| \int_{B_3(\eta) \setminus B_\delta (\eta)} \frac{|x-y|^\alpha}{|\xi -y|^{n+1-2s}} \int_{\frac{|\xi -y|^2}{t}}^\infty \tau^{n/2-s} e^{-\tau} d \tau d y \nonumber \\
& \leq & C_1 \, |x-\bar{x}| \int_{B_3(\eta) \setminus B_\delta (\eta)} \frac{|x-y|^\alpha}{|\xi -y|^{n+1-2s}} d \, y \nonumber \\
& \leq & C_2 \, \delta \, ( C_3 + \delta^{2s+\alpha-1}) \nonumber \\
& \leq & C_4 \, \delta^{2s+\alpha}, \; \; \mbox{ if } 2s +\alpha <1.
\label{3.11}
\end{eqnarray}
In the case when $2s +\alpha =1$, the last two lines of the above estimates becomes
$$ \leq C_2 \, \delta \, ( C_3 + |\log \delta| ) \leq C_4 \delta \, |\log \delta|.$$

Here, we have again exploited  the fact that, on $B_3(\eta)\setminus B_\delta (\eta)$, $|x-y|$ is equivalent to $|\xi -y|$.

On the second part, we estimate each term separately.
\begin{eqnarray}
& & \int_{B_\delta (\eta)} |x-y|^\alpha \int_0^t \frac{1}{(t-\tau)^{n/2+1-s}} e^{- \frac{|x-y|^2}{4(t-\tau)}} d \, y \, d \, \tau \nonumber \\
& \leq & C \int_{B_\delta (\eta)} \frac{|x-y|^\alpha}{|x-y|^{n-2s}} d y \nonumber \\
&\leq &  C \, \delta^{2s+\alpha}.
\label{3.12}
\end{eqnarray}

Similarly
\begin{eqnarray}
& & \int_{B_\delta (\eta)} |x-y|^\alpha \int_0^t \frac{1}{(t-\tau)^{n/2+1-s}} e^{- \frac{|\bar{x}-y|^2}{4(t-\tau)}} d \, y \, d \, \tau \nonumber \\
& \leq & C \int_{B_\delta (\eta)} \frac{|x-y|^\alpha}{|\bar{x}-y|^{n-2s}} d y \nonumber \\
& \leq & C_1 \, \delta^\alpha \int_{B_\delta (\eta)} \frac{1}{|\bar{x}-y|^{n-2s}} d y \nonumber \\
&\leq &  C_2 \, \delta^{2s+\alpha}.
\label{3.13}
\end{eqnarray}

Now applying (\ref{3.11}), (\ref{3.12}), (\ref{3.13}), and (\ref{3.14}), we arrive at (\ref{Ewx3}). This completes the proof of Lemma \ref{lem4.4}. \end{proof}

\subsubsection{Schauder estimate in the case $1 < 2s + \alpha  \leq 2$}

\begin{lemma} \label{lem4.5}
There exists a constant $C$ independent of $t \in (1,2)$ and $x, \bar{x} \in B_1(0)$, such that
\begin{equation} \label{Ewx4}
\left|\frac{\partial w}{\partial x_i} (x,t) - \frac{\partial w}{\partial x_i} (\bar{x}, t)\right| \leq \left\{ \begin{array}{ll}
C \, \|f\|_{C^\alpha(Q)} |x - \bar{x}|^{2s +\alpha -1}, & \mbox{ if } 1< 2s +\alpha < 2, \\
C \, \|f\|_{C^\alpha(Q)} |x - \bar{x}|  |\log \min\{|x-\bar{x}|, 1/2\}|, & \mbox{ if } 2s +\alpha = 2.
\end{array}
\right.
\end{equation}
\end{lemma}

\begin{proof}
Similar to (\ref{3.10}), by symmetry, we have
\begin{equation} \label{3.10b}
\int_{B_3 (\eta)} e^{- \frac{|x-y|^2}{4(t-\tau)}} (x_i - y_i) d \, y = \int_{B_3 (\eta)} e^{- \frac{|\bar{x}-y|^2}{4(t-\tau)}} (\bar{x}_i -y_i) d \, y.
\end{equation}
Therefore
\begin{eqnarray}
& & \left|\frac{\partial w}{\partial x_i}(x,t) - \frac{\partial w}{\partial x_i} (\bar{x}, t)\right| \nonumber \\
& = & C \, \left| \int_{B_3(\eta)} \int_0^t \frac{  f_Q (y,\tau) - f_Q(x,\tau) }{(t-\tau)^{n/2+2-s}}  \left(  e^{- \frac{|x-y|^2}{4(t-\tau)}} (x_i-y_i) -
e^{- \frac{|\bar{x}-y|^2}{4(t-\tau)}} (\bar{x}_i-y_i)\right) d \, y \, d \, \tau \right| \nonumber \\
& \leq & C_1 \int_{B_3(\eta)} \int_0^t \frac{  \|f\|_{C_x^\alpha(Q)} |x-y|^\alpha }{(t-\tau)^{n/2+2-s}} \left|  e^{- \frac{|x-y|^2}{4(t-\tau)}} (x_i-y_i) -
e^{- \frac{|\bar{x}-y|^2}{4(t-\tau)}} (\bar{x}_i-y_i) \right| d \, y \, d \, \tau
\label{3.20}
\end{eqnarray}

We first estimate the integral on $B_3(\eta) \setminus B_\delta (\eta)$. We apply the Mean Value Theorem to the difference in the integral on the left hand side of (\ref{3.20}).

\begin{eqnarray}
& & \int_{B_3(\eta) \setminus B_\delta (\eta)} \int_0^t \frac{|x-y|^\alpha }{(t-\tau)^{n/2+2-s}} \left|  e^{- \frac{|x-y|^2}{4(t-\tau)}} (x_i-y_i) -
e^{- \frac{|\bar{x}-y|^2}{4(t-\tau)}} (\bar{x}_i-y_i) \right| d \, y \, d \, \tau \nonumber \\
& \leq & C  \int_{B_3(\eta) \setminus B_\delta (\eta)} \int_0^t \frac{|x-y|^\alpha}{(t-\tau)^{n/2+2-s}} e^{-\frac{|\xi-y|^2}{4(t-\tau)}} \left( \frac{|\xi-y|^2}{t-\tau} + 1 \right) |x - \bar{x}| \, d y \, d \tau \nonumber \\
& \leq & C_1 \, \delta \int_{B_3(\eta) \setminus B_\delta (\eta)} \frac{|x-y|^\alpha}{|\xi-y|^{n+2-2s}} \int_{\frac{|\xi-y|}{t}}^\infty ( \tau^{n/2-s} + \tau^{n/2-1-s} ) e^{-\tau} d \tau d y \nonumber \\
& \leq & C_2 \, \delta \int_{B_3(\eta) \setminus B_\delta (\eta)} \frac{|x-y|^\alpha}{|\xi-y|^{n+2-2s}} d y \nonumber \\
& \leq & C_3 \, \delta ( C_4 + C_5 \delta^{2s+\alpha -2}) \nonumber \\
& \leq & C_6 \, \delta^{2s+\alpha-1}, \;\; \mbox{ if } 1< 2s +\alpha < 2.
\label{3.21}
\end{eqnarray}
In the case $2s + \alpha =2$, the last two lines in the above estimates becomes
\begin{equation} \label{3.21a}
 \leq  C_3 \, \delta \, ( C_4 + C_5 |\log \delta|) \leq C_6 \, \delta \, |\log \delta|.
 \end{equation}

For the integral on $B_\delta (\eta)$, we estimate each term separately.

\begin{eqnarray}
& & \int_{B_\delta (\eta)} \int_0^t \frac{ |x-y|^\alpha }{(t-\tau)^{n/2+2-s}}  \left| e^{- \frac{|x-y|^2}{4(t-\tau)}} (x_i-y_i) \right|  d \, y \, d \, \tau   \nonumber \\
& \leq & C  \int_{B_\delta (\eta)} |x-y|^\alpha \int_0^t e^{-\frac{|x-y|^2}{4(t-\tau)}} \frac{|x-y|}{(t-\tau)^{n/2+2-s}} \, d \, \tau \, d\, y \nonumber \\
& = & C \int_{B_\delta (\eta)} \frac{|x-y|^{\alpha}}{|x-y|^{n+1-2s}} \int_{\frac{|x-y|^2}{t}} ^\infty \tau^{n/2-s} e^{-\tau}  \, d \, \tau \, d\, y \nonumber \\
& \leq & C_1 \int_{B_\delta (\eta)} \frac{|x-y|^{\alpha}}{|x-y|^{n+1-2s}} \, d\, y \nonumber \\
& \leq & C_2 \, \delta^{2s+\alpha-1}.
\label{3.22}
\end{eqnarray}

Now if we apply the similar argument to the second part of (\ref{3.20}), what we obtain is
\begin{eqnarray}
& & \int_{B_\delta (\eta)} \int_0^t \frac{ |x-y|^\alpha }{(t-\tau)^{n/2+2-s}}  \left| e^{- \frac{|\bar{x}-y|^2}{4(t-\tau)}} (\bar{x}_i-y_i) \right|  d \, y \, d \, \tau  \nonumber \\
& \leq & C \int_{B_\delta (\eta)} \frac{|x-y|^{\alpha}}{|\bar{x}-y|^{n+1-2s}} \, d\, y.
\label{3.23}
\end{eqnarray}
Here in $B_\delta (\eta)$, $|x-y|$ is not equivalent to $|\bar{x}-y|$, and in our case $s \geq 1/2$, the integral on the last line of (\ref{3.23}) diverges. For this reason, we need to modify the right hand side of (\ref{3.20}), so that the term $|x-y|^\alpha$ in (\ref{3.23}) will become $|\bar{x}-y|^\alpha$.

For simplicity of notation, at this moment, we denote
$$ F(x,y) = e^{-\frac{|x-y|^2}{4(t-\tau)}} \;\; f(y) = f_Q(y, \tau) \;\; g(x) = \int_{B_\delta (\eta)} f(y) F(x,y) d \, y.$$
Although $F$, $f$, and $g$ are also functions of $t$ and $\tau$, at this moment we fix $t$ and $\tau$.

We first estimate the difference
$$ \frac{\partial g}{\partial x_i}(x) - \frac{\partial g}{\partial x_i}(\bar{x}).$$
We have
\begin{eqnarray*}
& & \frac{\partial g}{\partial x_i}(x) = \\
& & \int_{B_\delta (\eta)} f(y) \frac{\partial F}{\partial x_i} (x,y) d \, y -  \int_{B_\delta (\eta)} f(x) \frac{\partial F}{\partial x_i} (x,y) d \, y +  \int_{B_\delta (\eta)} f(x) \frac{\partial F}{\partial x_i} (x,y) d \, y.
\end{eqnarray*}
Similarly,
\begin{eqnarray*}
& & \frac{\partial g}{\partial x_i}(\bar{x}) = \\
& & \int_{B_\delta (\eta)} f(y) \frac{\partial F}{\partial x_i} (\bar{x},y) d \, y -  \int_{B_\delta (\eta)} f(\bar{x}) \frac{\partial F}{\partial x_i} (\bar{x},y) d \, y +  \int_{B_\delta (\eta)} f(\bar{x}) \frac{\partial F}{\partial x_i} (\bar{x},y) d \, y.
\end{eqnarray*}
It follows that
\begin{eqnarray}
& & \left|  \frac{\partial g}{\partial x_i}(x) - \frac{\partial g}{\partial x_i}(\bar{x}) \right| \nonumber \\
& \leq & \int_{B_\delta (\eta)} | f(y) - f(x)| \frac{\partial F}{\partial x_i} (x,y) d \, y + \int_{B_\delta (\eta)} |f(y)- f(\bar{x})| \frac{\partial F}{\partial x_i} (\bar{x},y) d \, y \nonumber \\
& + & \left| \int_{B_\delta (\eta)} f(x) \frac{\partial F}{\partial x_i} (x,y) d \, y - \int_{B_\delta (\eta)} f(\bar{x}) \frac{\partial F}{\partial x_i} (\bar{x},y) d \, y \right| \nonumber \\
& \equiv & I_1 + I_2 + I_3.
\label{3.24}
\end{eqnarray}

Similar to the above arguments, we derive
\begin{equation}
\int_0^t \frac{1}{(t-\tau)^{n/2+1-s}} \, I_1 \, d \tau \leq C \, \|f\|_{C_x^\alpha (Q)} \delta^{2s+\alpha-1},
\label{3.25}
\end{equation}
and
\begin{equation}
\int_0^t \frac{1}{(t-\tau)^{n/2+1-s}} \, I_2 \, d \tau \leq C \, \|f\|_{C_x^\alpha (Q)} \delta^{2s+\alpha-1}.
\label{3.26}
\end{equation}

To estimate $I_3$, we employ the symmetry
$$ \int_{B_\delta (\eta)} \frac{\partial F}{\partial x_i} (x,y) d \, y = \int_{\partial B_\delta (\eta)} F (x,y) \nu_i (y)  d \sigma_y
= \int_{\partial B_\delta (\eta)} F (\bar{x},y) \nu_i (y)  d \sigma_y =  \int_{B_\delta (\eta)} \frac{\partial F}{\partial x_i} (\bar{x},y) d \, y$$
where $\nu_i(y)$ is the $ith$ component of the unit outward normal on $\partial B_\delta (\eta)$. Therefore
\begin{eqnarray}
& & \int_0^t \frac{1}{(t-\tau)^{n/2+1-s}} \, I_3 \, d \tau  \nonumber \\
& \leq & C \, \|f\|_{C_x^\alpha (Q)} |x-\bar{x}|^\alpha \int_{\partial B_\delta (\eta)} \int_0^t \frac{1}{(t-\tau)^{n/2+1-s}} e^{-\frac{|x-y|^2}{4(t-\tau)}} d \, \tau \, d \sigma_y \nonumber \\
& \leq & C_1 \, \|f\|_{C_x^\alpha (Q)} |x-\bar{x}|^\alpha \int_{\partial B_\delta (\eta)} \frac{1}{|x-y|^{n-2s}} d \sigma_y \nonumber \\
& \leq & C_2 \, \|f\|_{C_x^\alpha (Q)} \delta^{2s+\alpha-1}.
\label{3.27}
\end{eqnarray}

Now by \eqref{3.25}, \eqref{3.26}, \eqref{3.27}, and \eqref{3.24}, we arrive at
\begin{equation} \label{3.28}
\left| \int_{B_\delta(\eta)} \int_0^t \frac{  f_Q (y,\tau) }{(t-\tau)^{n/2+2-s}}  \left(  e^{- \frac{|x-y|^2}{4(t-\tau)}} (x_i-y_i) -
e^{- \frac{|\bar{x}-y|^2}{4(t-\tau)}} (\bar{x}_i-y_i)\right) d \, y \, d \, \tau \right| \leq C \, \|f\|_{C_x^\alpha (Q)} \delta^{2s+\alpha-1}.
\end{equation}

Finally, substituting the results in \eqref{3.21}, \eqref{3.21a}, and \eqref{3.28} into  \eqref{3.20}, we arrive at \eqref{Ewx4}.
This completes the proof of Lemma \ref{lem4.5}. \end{proof}
\medskip

\begin{remark}
By symmetry, \eqref{3.10b} holds if $B_3(\eta)$ is replaced by $B_\delta(\eta)$. Hence in (\ref{3.28}), the term $f_Q(y,\tau)$ can be replaced by $f_Q(y,\tau) - f_Q(x,\tau)$, or $f_Q(y,\tau) - f_Q(\bar{x},\tau)$, while the results are the same.
\end{remark}

\subsubsection{Schauder estimate in the case $2 < 2s + \alpha < 3$}

Employing a similar argument as in the previous subsection, one can obtain
\begin{lemma} \label{lem4.6}
 There exists constant $C$ independent of $t \in (1,2)$ and $x, \bar{x} \in B_1(0)$, such that
\begin{equation} \label{3.30a}
\left| \frac{\partial^2 w}{\partial x_i \partial x_j}(x,t) - \frac{\partial^2 w}{\partial x_i \partial x_j}(\bar{x},t) \right| \leq C
\|f\|_{C^\alpha(Q)} |x - \bar{x}|^{2s +\alpha -2}, \;\; \mbox{ if } 2 < 2s + \alpha < 3.
\end{equation}
\end{lemma}

\subsection{Regularity of w in t}

For each fixed $x$, we derive H\"{o}lder and Schauder estimates of $w(x,t)$ in time variable $t$.

\subsubsection{H\"{o}lder estimate}

We prove
\begin{lemma} \label{lem4.7}  There exists a constant $C$ independent of  $x \in B_1(0)$ and $t, \bar{t} \in (1,2)$,  such that
\begin{equation}
| w(x,t) - w(x, \bar{t}) | \leq C \, \|f\|_{L^\infty (Q)} |t - \bar{t}|^s.
\label{Ewt1}
\end{equation}
\end{lemma}

\begin{proof} Let $\chi_{(a,b)}$ be the characteristic function on interval $(a,b)$.

Assume $\bar{t} < t$ and $|t-\bar{t}|=\delta.$

We have
\begin{eqnarray}
&  & | w(x,t) - w(x, \bar{t}) | \nonumber \\
& \leq & \int_0^2 \int_{B_2(0)} \left| G(x-y, t-\tau) \chi_{(0,t)} - G(x-y, \bar{t} - \tau) \chi_{(0, \bar{t})} \right| f(y,\tau)  d \, y \, d \, \tau \nonumber \\
& \leq & \|f\|_{L^\infty(Q)} \int_0^2 \int_{B_2(0)} \left| G(x-y, t-\tau)- G(x-y, \bar{t} - \tau) \right| \chi_{(0,t)} d \, y \, d \, \tau \nonumber \\
& + & \|f\|_{L^\infty (Q)} \int_0^2 \int_{B_2(0)} G(x-y, \bar{t} - \tau) [ \chi_{(0,t)} - \chi_{(0,\bar{t})} ] d \, y \, d \, \tau  \nonumber \\
& = & \|f\|_{L^\infty (Q)} ( I_1 + I_2 ).
\label{3.30}
\end{eqnarray}

First estimate $I_1$. Divide the interval $(0,t)$ into two parts, $(0, \bar{t}-\delta)$ and $(\bar{t}-\delta, t)$. For the first part, we apply the Mean Value Theorem and let $\xi$ be some point between $\bar{t}$ and $t$. We derive
\begin{eqnarray}
& & \int_0^{\bar{t}-\delta}  \int_{B_2(0)} \left| G(x-y, t-\tau)- G(x-y, \bar{t} - \tau) \right| d \, y \, d \, \tau \nonumber \\
& \leq & C \, |t-\bar{t}| \,  \int_0^{\bar{t}-\delta}  \int_{B_2(0)} \frac{1}{(\xi -\tau)^{n/2 +2-s}} e^{-\frac{|x-y|^2}{4(\xi-\tau)}} \left[ 1 + \frac{|x-y|^2}{4(\xi-\tau)} \right] d \, y \, d \, \tau \nonumber \\
& \leq & C \, \delta \,  \int_0^{\bar{t}-\delta}  \frac{1}{(\xi -\tau)^{2-s}}  \int_{\mathbb{R}^n} e^{-|z|^2} (1+ |z|^2) \, d \, z \, d \, \tau \nonumber \\
& \leq & C_1 \, \delta ( \delta^{s-1} + 1) \nonumber \\
& \leq & C_2 \, \delta^s.
\label{3.31}
\end{eqnarray}

For the integral on $(\bar{t}-\delta, t)$, we estimate two terms separately.
\begin{eqnarray}
& & \int_{\bar{t}-\delta}^t \int_{B_2(0)} | G(x-y, t-\tau) | d \, y \, d \, \tau \nonumber \\
& \leq & \int_{\bar{t}-\delta}^t  \frac{1}{(t -\tau)^{n/2 +1-s}}\int_{\mathbb{R}^n} e^{-\frac{|x-y|^2}{4(t-\tau)}} d \, y \, d \, \tau \nonumber \\
& \leq & \int_{\bar{t}-\delta}^t  \frac{1}{(t -\tau)^{n/2 +1-s}}\int_{\mathbb{R}^n} e^{-|z|^2} d \, z \, d \, \tau \nonumber \\
& \leq & C \, \int_{\bar{t}-\delta}^t  \frac{1}{(t -\tau)^{1-s}} \, d \, \tau \nonumber \\
& \leq & C_1 \, \delta^s.
\label{3.32}
\end{eqnarray}

Similarly
\begin{equation}
\int_{\bar{t}-\delta}^t \int_{B_2(0)} | G(x-y, \bar{t}-\tau) | d \, y \, d \, \tau \leq  C \, \int_{\bar{t}-\delta}^t  \frac{1}{|\bar{t} -\tau|^{1-s}} \, d \, \tau \leq C_1 \delta^s.
\label{3.33}
\end{equation}

It follows from (\ref{3.31}), (\ref{3.32}), and (\ref{3.33}) that
\begin{equation}
I_1 \leq C \, \delta^s.
\label{3.34}
\end{equation}

Then we estimate $I_2$.
\begin{eqnarray}
I_2 &=& \int_{\bar{t}}^t  \frac{1}{|\bar{t} -\tau|^{n/2 +1-s}}\int_{B_2(0)} e^{-\frac{|x-y|^2}{4|\bar{t}-\tau|}} d \, y \, d \, \tau \nonumber \\
& \leq & C \,  \int_{\bar{t}}^t  \frac{1}{|\bar{t} -\tau|^{1-s}} \, d \, \tau \nonumber \\
& \leq & C_1 \, \delta^s.
\label{3.35}
\end{eqnarray}

Now from (\ref{3.30}), (\ref{3.34}), and (\ref{3.35}), we arrive at (\ref{Ewt1}). This completes the proof of Lemma \ref{lem4.7}.
\end{proof}

\subsubsection{Schauder estimate}

Assume that for all $x \in B_2(0)$, $f(t,x)$ is uniformly $\alpha/2$-H\"{o}lder continuous in $t$ on $(0,3)$, and we simply denote this norm by $\|f\|_{C_t^{\alpha/2}(Q)}$.
We prove
\begin{lemma} \label{lem4.9}
There exists a constant $C$ independent of  $x \in B_1(0)$ and $t, \bar{t} \in (1,2)$, such that
\begin{equation}
|w(x, t) - w(x, \bar{t})| \leq \left\{\begin{array}{ll} C \, \|f\|_{C_t^{\alpha/2}(Q)} |t-\bar{t}|^{s+\alpha/2}, & \mbox{ if } s + \alpha/2 <1, \\
C \, \|f\|_{C_t^{\alpha/2}(Q)} |t-\bar{t}| \, |\log \min\{|t-\bar{t}|, 1/2\}| & \mbox{ if } s + \alpha/2 =1.
\end{array}
\right.
\label{Ewt2}
\end{equation}
\end{lemma}

\begin{proof}
\begin{eqnarray}
&  & | w(x,t) - w(x, \bar{t}) | \nonumber \\
& = & \int_0^2 \int_{B_2(0)} \left| G(x-y, t-\tau) \chi_{(0,t)} - G(x-y, \bar{t} - \tau) \chi_{(0, \bar{t})} \right| f(y,\tau)-f(y,t)|  d \, y \, d \, \tau \nonumber \\
& + & \int_0^2 \int_{B_2(0)} \left| G(x-y, t-\tau) \chi_{(0,t)} - G(x-y, \bar{t} - \tau) \chi_{(0, \bar{t})} \right| |f(y,t)|  d \, y \, d \, \tau \nonumber \\
& \leq & \|f\|_{C_t^{\alpha/2}(Q)} \int_0^2 \int_{B_2(0)} \left| G(x-y, t-\tau) \chi_{(0,t)} - G(x-y, \bar{t} - \tau) \chi_{(0, \bar{t})} \right|t-\bar{t}|^{\alpha/2}  d \, y \, d \, \tau \nonumber \\
& + & \|f\|_{L^\infty (Q)} \int_0^2 \int_{B_2(0)} \left| G(x-y, t-\tau) \chi_{(0,t)} - G(x-y, \bar{t} - \tau) \chi_{(0, \bar{t})} \right|  d \, y \, d \, \tau \nonumber \\
& = & \|f\|_{C_t^{\alpha/2}(Q)} I_1 + \|f\|_{L^\infty (Q)} I_2 .
\label{3.36}
\end{eqnarray}

 Similar to the arguments in (\ref{3.31}), (\ref{3.32}), and (\ref{3.33}, we arrive at
\begin{equation}
I_1 \leq \left\{\begin{array}{ll} C \, \delta^{s + \alpha/2}, & \mbox{ if } s + \alpha /2 <1, \\
C \, \delta \, |\log \delta|, & \mbox{ if } s + \alpha /2 =1.
\end{array}
\right.
\label{3.37}
\end{equation}

While in the case $s + \alpha /2 =1$,

Now what left is to estimate $I_2$. We have, by making change of the variable,
$$
\int_0^2 \int_{B_2(0)} G(x-y, t-\tau) \chi_{(0,t)} \, d \, y \, d \, \tau = \int_0^t \int_{B_2(0)} G(x-y, \tau) \, d \, y \, d \, \tau.
$$
Similarly,
$$
\int_0^2 \int_{B_2(0)} G(x-y, \bar{t}-\tau) \chi_{(0,\bar{t})} \, d \, y \, d \, \tau = \int_0^{\bar{t}} \int_{B_2(0)} G(x-y, \tau) \, d \, y \, d \, \tau.
$$
Consequently,
\begin{equation}
I_2 \leq \int_{\bar{t}}^t G(x-y, \tau) \,d \, y \, d \, \tau \leq C \, \int_{\bar{t}}^t \frac{1}{\tau^{1-s}} d \, \tau = C_1 (t^s - \bar{t}^s) \leq C_2 \, |t-\bar{t}|.
\label{3.38}
\end{equation}

Combining (\ref{3.37}), (\ref{3.38}), and (\ref{3.36}), we establish (\ref{Ewt2}) and thus complete the proof of Lemma \ref{lem4.9}.
\end{proof}

\subsection{Conclusion of the estimates on w(x,t)}

Now Lemmas \ref{lem4.1}, \ref{lem4.2}, \ref{lem4.3}, and \ref{lem4.7} imply Theorem \ref{thm3.1} while Lemmas \ref{lem4.4}, \ref{lem4.5}, \ref{lem4.6}, and \ref{lem4.9} lead to Theorem \ref{thm3.2}. This concludes the section.

\section{A Priori Estimates of Solutions}

The regularity results established in the previous sections make it possible to carry out the blowing up and re-scaling arguments.
Now in this section, we will apply these arguments to obtained a priori estimates of nonnegative solutions for a family of master equations in the whole space.

We first consider
\begin{equation}
(\partial_t - \Delta)^s u(x,t) = f(x, u(x,t)), \;\; (x,t) \in \mathbb{R}^n \times \mathbb{R}.
\label{4.1}
\end{equation}

\begin{theorem} \label{thm4.1}
Assume that $u$ is a nonnegative solution of (\ref{4.1}) and

(f1) $f(x, \tau)$ is uniformly H\"{o}lder continuous with respect to $x$ and continuous  with respect to $\tau$.

 (f2) $0\leq f(x, \tau) \leq C_0(1+\tau^p)$  uniformly for all $x$, and
 $$
 \mathop{\lim}\limits_{\tau \to \infty}\frac{f(x, \tau)}{\tau^p}=K(x), \,\, 1<p<\frac{n+2}{n+2-2s},
 $$
 where $K(x)\in (0, \infty)$ is uniformly continuous and $ \mathop{\lim}\limits_{|x|\to \infty} K(x)=\bar C\in (0, \infty)$.
 Then there exists a constant $C$, such that
 \begin{equation} \label{4.2}
 u(x,t) \leq C, \,\,\, \forall \, (x,t) \in \mathbb{R}^n \times \mathbb{R}
 \end{equation}
 for all nonnegative solutions $u$ of \eqref{4.1}.
 \end{theorem}
 The above a priori estimate holds for all $0<s<1$.
 \medskip

Then in the case $1/2 <s <1$, we study more general equation
 \begin{equation} \label{4.2a}
 (\partial_t - \Delta)^s u(x,t) = b(x)|\nabla_x u(x,t)|^q + f(x, u(x,t)), \;\; (x,t) \in \mathbb{R}^n \times \mathbb{R}
 \end{equation}
 and establish a priori estimates for its non-negative solutions. Here $\nabla_x u$ is the gradient of $u$ with respect to $x$.

\begin{theorem} \label{thm4.2} Assume that
  $b(x):  \mathbb{R}^n \to \mathbb{R}$  is nonnegative, uniformly bounded and  H\"{o}lder continuous.

If $s>\frac{1}{2},$ then there exists a positive constant $C$ such that
$$
u(x,t)+|\nabla_x u|^{\frac{2s}{2s+p-1}}(x,t)\leq C,  \;\;\; \forall \, (x,t) \in \mathbb{R}^n \times \mathbb{R},
$$
for all nonnegative solutions $u$.
\end{theorem}

\subsection{Proof of Theorem \ref{thm4.1}}

\begin{proof}

We will prove Theorem \ref{thm4.1} by a contradiction argument. Suppose (\ref{4.2}) is violated, then there exist a sequence of solutions $u_k$ of (\ref{4.1}) and a sequence of points  $(x^k, t_k)$ such that
$$u_k(x_k, t_k) \to \infty, \mbox{ as } k \to \infty. $$

For each fixed $R$, let
$$
R_k:=2 R u_k^{-\frac{p-1}{2s}}(x_k, t_k).
$$

Denote $X =(x,t)$ and $X_k = (x_k, t_k)$.
Define a function
$$
S_k(X)=u_k(X)(R_k-|X-X_k|)^{\frac{2s}{p-1}},\,\, X \in B_{R_k}(X_k).
$$
Then there exists $A_k \in B_{R_k}(X_k)$ such that
$$
S_k(A_k)=\mathop{\max}\limits_{B_{R_k}(X_k)} S_k(X).
$$
It follows that
\begin{eqnarray}\label{eq4-3}
u_k(X)\leq u_k(A_k)\frac{(R_k-|A_k-X_k|)^{\frac{2s}{p-1}}}{(R_k-|X-X_k|)^{\frac{2s}{p-1}}},\,\, X \in B_{R_k}(X_k).
\end{eqnarray}

Set
$$
\lambda_k=(u_k(A_k))^{-\frac{p-1}{2s}}.
$$
Taking $X=X_k$ in \eqref{eq4-3} yields
 \begin{eqnarray}\label{eq4-4}
\frac{u_k(X_k)}{u_k(A_k)}\leq \frac{(R_k-|A_k-X_k|)^\frac{2s}{p-1}}{R_k^{\frac{2s}{p-1}}},
\end{eqnarray}
which implies  that
$$
u_k(X_k)\leq u_k(A_k).
$$
Using the definition of $R_k$ and $\lambda_k$, one gets from \eqref{eq4-4} that
 \begin{eqnarray}\label{eq4-5}
2R \lambda_k\leq R_k-|A_k-X_k|.
\end{eqnarray}

In addition, one can derive
\begin{eqnarray}\label{eq4-6}
R_k-|A_k-X_k|\leq 2(R_k-|X-X_k|),\,\, X\in B_{R \lambda_k}(A_k).
\end{eqnarray}
Indeed, for any $X \in B_{R \lambda_k}(A_k)$, one has
$$
|X-X_k|\leq |X-A_k|+|A_k-X_k|\leq R\lambda_k +R_k-2R\lambda_k \leq R_k
$$
due to \eqref{eq4-5}. It follows that  $X\in B_{R_k}(X_k)$
and thus
$B_{R\lambda_k}(A_k)\subset B_{R_k}(X_k).$ Now for any $X\in B_{R\lambda_k}(A_k)$, using \eqref{eq4-5}, one derives
\begin{eqnarray*}
\begin{aligned}
	2(R_k-|X-X_k|)&\geq 2(R_k-(|X_k-A_k|+|X-A_k|))\\
	&\geq 2(R_k -(|X_k-A_k|+R\lambda_k))\\
	&\geq R_k -|X_k-A_k|+2R\lambda_k-2R\lambda_k\\
	&\geq R_k -|X_k-A_k|.
\end{aligned}
\end{eqnarray*}
This implies \eqref{eq4-6}.

The combination of  \eqref{eq4-3} with \eqref{eq4-6} yields
\begin{eqnarray}\label{eq4-7}
u_k(X)\leq u_k(A_k) 2^\frac{2s}{p-1},\,\, x\in B_{R \lambda_k}(A_k).
\end{eqnarray}

Let $\bar{R} = R/\sqrt{n+1}$ and denote $A_k = (\bar{x}_k, \bar{t}_k)$,  then it is obvious that the parabolic cylinder
$$ Q_{\bar{R} \lambda_k}(\bar{x}_k, \bar{t}_k)  = \{(x,t) \mid |x-\bar{x}_k| < \bar{R} \lambda_k, \;\; |t-\bar{t}|< \bar{R}^2 \lambda_k^2 \}$$ is contained in
$B_{R\lambda_k}(A_k)$ for sufficiently large $k$.

 We now re-scale the solutions as
 $$
   v_k(x,t)=\frac{1}{u_k(\bar{x}_k, \bar{t}_k)}u_k(\lambda_kx+\bar{x}_k, \lambda_k^2 t + \bar{t}_k), \,\, (x,t) \in Q_{\bar{R}}(0,0).
 $$
Then it follows from (\ref{4.1}) that $v_k$ solves
\begin{eqnarray}\label{eq4-8}
(\partial_t -\Delta)^s v_k(x,t) &=&
\lambda_k^\frac{2sp}{p-1} f(\lambda_kx+\bar{x}_k, \lambda_k^{-\frac{2s}{p-1}}v_k(x,t)) \nonumber \\
&:=&F_k(x, v_k(x,t)),\,\, (x,t)\in Q_{\bar{R}}(0,0).
\end{eqnarray}
Moreover, it follows from the definition of $v_k$ and \eqref{eq4-7} that
\begin{eqnarray}\label{eq4-9}
v_k(0,0)=1
\end{eqnarray}
and
\begin{eqnarray}\label{eq4-10}
v_k(x,t)\leq 2^\frac{2s}{p-1},\,\, (x,t)\in Q_{\bar{R}}(0,0).
\end{eqnarray}
Then due to \eqref{eq4-10} and the assumptions (f1) and (f2),  one derives
\begin{eqnarray*}
\begin{aligned}
F_k(x, v_k(x,t))
=&\lambda_k^\frac{2sp}{p-1} f(\lambda_kx+\bar{x}_k, \lambda_k^{-\frac{2s}{p-1}}v_k(x,t))\\
\leq &C_0\lambda_k^\frac{2sp}{p-1}\left(1+\lambda_k^{-\frac{2sp}{p-1}}v_k^p(x,t)\right)\\
\leq &C_0\lambda_k^\frac{2sp}{p-1}+C_0 v_k^p(x,t)\\
\leq &C_1.
\end{aligned}
\end{eqnarray*}
Applying Theorem \ref{mthmu1} to  \eqref{eq4-8} yields
\begin{eqnarray*}
\|v_k\|_{C_{x,t}^{2s-\varepsilon, s}(Q_{\bar{R}/2}(0,0))} \leq C_2
\end{eqnarray*}
for small $0<\varepsilon.$
It follows that $F_k(x, v_k(x,t))$ is H\"{o}lder continuous
in $Q_{\bar{R}}(0,0)).$ Then applying Theorem \ref{mthmu2}, we conclude that there exist some $\varepsilon >0$
and $C_3$ independent of $k$ such that
\begin{eqnarray}\label{eq4-11}
\|v_k\|_{C_{x,t}^{2s+\varepsilon, s + \varepsilon}(Q_{\bar{R}/4}(0,0))}\leq C_3.
\end{eqnarray}

 The above uniform regularity estimates holds in $Q_{\bar{R}/4}(0,0)$ for any $\bar{R}$. Let $\bar{R} \to \infty$, we conclude there exists
 a converging subsequence of $\{v_k(x,t)\}$ (stilled denoted by $\{v_k(x,t)\}$) that converges pointwise in
 $\mathbb{R}^n \times \R$ to a function $v(x,t).$
Taking limit in \eqref{eq4-9} yields
\begin{eqnarray}\label{eq4-12}
v(0,0)=1,	
\end{eqnarray}
and
\begin{eqnarray}
v(x,t)\leq 2^\frac{2s}{p-1},\,\, (x,t)\in \R^n\times\R.
\end{eqnarray}
 Moreover, $\{v_k(x,t)\}$ converges to $v(x,t)$ locally in $C_{x,t}^{2s+\varepsilon, s + \varepsilon}(\R^n \times \R)$, and hence  $(\partial_t -\Delta)^s v_k$ converges pointwise in $\mathbb{R}^n \times \R$ and $(\partial_t-\Delta)^sv(x,t)$ is uniformly bounded in $\R^n\times\R.$

Applying a result in \cite{CGL} we obtain
\begin{eqnarray}\label{eq4-13}
\lim_{k \to \infty} (\partial_t -\Delta)^s v_k(x,t)= (\partial_t-\Delta)^s v(x,t)-b, \;\; \forall \, (x,t) \in \mathbb{R}^n \times \R
\end{eqnarray}
with some nonnegative constant $b$.

In addition, for any fixed $x\in \mathbb{R}^n,$ if $\{\lambda_k x+a_k\}$ is bounded, we assume that $\lambda_k x+a_k \to \bar x$  by extracting a further
subsequence. Therefore, the assumption (f2) implies
\begin{eqnarray}\label{eq4-14}
\begin{aligned}
&\lambda_k^{\frac{2sp}{p-1}} f(\lambda_k x+a_k, \lambda_k^{-\frac{2s}{p-1}} v_k(x,t))\\
= & (v_k(x,t))^p\frac{f(\lambda_k x+a_k, \lambda_k^{-\frac{2s}{p-1}} v_k(x,t))}{\left(\lambda_k^{-\frac{2s}{p-1}} v_k(x,t)\right)^p}\\
\to& K(\bar x) v^p(x,t)\,\, \mbox{ as }\,\, k\to \infty.
\end{aligned}
\end{eqnarray}
This is  still valid in the case $|\lambda_k x+a_k| \to \infty$ with $|\bar x|= \infty$ due to (f2).

Consequently,  it follows from \eqref{eq4-13} and \eqref{eq4-14} that $v(x,t)$ is a solution of
\begin{eqnarray}\label{eq4-15}
(\partial_t-\Delta)^s v(x,t)=K(\bar x) v^p(x,t) + b,\,\, x\in \mathbb{R}^n.
\end{eqnarray}

We show that $b$ must be $0.$

Let $$F_R(x,t) = \left\{\begin{array}{ll} K(\bar{x})v^p(x,t) + b, & x \in Q_R(0,0)\\
0, & x \in Q_R^C(0,0).
\end{array}
\right.
$$
$$
v_R(x,t)=\int_{-\infty}^t \int_{\mathbb{R}^n} F_R(y,\tau) G(x-y, t-\tau) dy \, d \tau.
$$

Then it is easy to see that
$$
\left\{\begin{array}{ll} (\partial_t-\Delta)^s v_R(x,t)= F_R(x,t), \;\;  (x,t) \in \mathbb{R}^n \times \R\\
 v_R(x,t) \to 0, \; \mbox{ as } |x| \to \infty \mbox{ or as } t \to -\infty.
\end{array}
\right.
$$

Denote
$$
w_R(x,t)=v(x,t)-v_R(x,t), \,\, (x,t) \in \mathbb{R}^n \times \R.
$$
Then
\begin{eqnarray*}
\begin{cases}
(-\Delta)^s w_R(x,t)\geq 0,& (x,t) \in \mathbb{R}^n \times \R,\\
\lim_{|x| \to \infty \mbox{ or } t \to -\infty}  w_R(x,t) \geq 0.
\end{cases}
\end{eqnarray*}
By the maximum principle, we have
$$
w_R(x,t) \geq 0,\,\, (x,t) \in \mathbb{R}^n \times \R.
$$

Now let $R \to \infty$, we arrive at
$$
v(x,t)\geq \int_{-\infty}^t \int_{\mathbb{R}^n} G(x-y, t-\tau) \left(K(\bar x) v^p(y, \tau)+b \right)dy \, d \tau, \,\, x \in \mathbb{R}^n.
$$
If $b>0,$ then the integral on the right hand side is greater or equal to
$$ b \int_{-\infty}^t \int_{\mathbb{R}^n} G(x-y, t-\tau) dy \, d \tau$$
which is obviously divergent. This contradicts the bounded-ness of $v(x,t)$.
Therefore, we must have $b=0$ and
\begin{equation} \label{5.50}
(\partial_t-\Delta)^s v(x,t)=K(\bar x) v^p(x,t),\,\, (x,t) \in \mathbb{R}^n \times \R.
\end{equation}
Moreover, from \eqref{eq4-12} and  the maximum principle,
one derives that
$$v(x,t)>0,\,\, \forall (x,t) \in \mathbb{R}^n \times \R.$$
Now we arrive at a bounded positive solution $v$ of \eqref{5.50}. This contradicts the non-existence results in \cite{FP} (Theorem 1.1) if $1< p < \frac{n+2}{n+2-2s}$.

Now we  complete the proof of Theorem \ref{thm4.1}.

\end{proof}

\subsection{Proof of Theorem \ref{thm4.2}}

\begin{proof}
Now we assume that $1> s>\frac{1}{2}$ and derive a priori estimate for nonnegative solutions of equation (\ref{4.2a}).
We show that
\begin{eqnarray}\label{eq4-16}
u(x,t)+|\nabla_x u|^{\frac{2s}{2s+p-1}}(x,t)\leq C,\,\, \forall \, (x,t) \in \mathbb{R}^n \times \mathbb{R}
\end{eqnarray}
for all nonnegative solutions $u$. Again it will be acomplished by a blow-up and re-scaling argument.

If \eqref{eq4-16} is not valid, then there exist a sequence of solutions $u_k$,
$$
(\partial_t -\Delta)^s u_k(x,t)=b(x)|\nabla_x u_k|^q(x,t)+f(x, u_k(x,t)), \,\, \forall \, (x,t) \in \mathbb{R}^n \times \mathbb{R}
$$
with
$$
M_k(x_k, t_k):=u_k^{\frac{p-1}{2s}}(x_k, t_k)+|\nabla_x u_k|^{\frac{p-1}{2s+p-1}}(x_k, t_k) \to \infty\,\, \mbox{as} \,\, k\to \infty.
$$

Again, for simplicity of notation, denote $X =(x,t)$ and $X_k = (x_k, t_k)$.

For each fixed $R$, let
$$
R_k:=2 R M_k^{-1}(X_k).
$$

To start the rescaling procedure and  obtain upper bounds for solutions, we consider the following function
$$
 S_k(X)=\left(M_k(X)(R_k-|X-X_k|)\right)^\frac{2s}{p-1},\,\, X \in B_{R_k}(X_k).
$$
Then there exists a point $A_k \in B_{R_k}(X_k)$  such that
$$
S_k(A_k)=\mathop{\max}\limits_{B_{R_k}(X_k)} S_k(X).
$$
It follows that
\begin{eqnarray}\label{eq4-17}
M_k(X)\leq M_k(A_k)\frac{R_k-|A_k-X_k|}{R_k-|X-X_k|},\,\, X\in B_{R_k}(X_k).
\end{eqnarray}

Denote
$$
\lambda_k=(M_k(A_k))^{-1}.
$$
Taking $X=X_k$ in \eqref{eq4-17}, it yields
 \begin{eqnarray}\label{eq4-18}
\frac{M_k(X_k)}{M_k(A_k)}\leq \frac{R_k-|A_k-X_k|}{R_k},
\end{eqnarray}
which implies that
$$
M_k(X_k)\leq M_k(A_k).
$$
It follows from  the definition of $R_k$ and $\lambda_k$ and  \eqref{eq4-18} that
 \begin{eqnarray}\label{eq4-19}
2R\lambda_k\leq R_k-|A_k-X_k|.
\end{eqnarray}

In addition, using \eqref{eq4-19} and by a similar argument as in deriving \eqref{eq4-6}, one derives
\begin{eqnarray}\label{eq4-20}
R_k-|A_k-X_k|\leq 2(R_k-|X-X_k|),\,\, X\in B_{R \lambda_k}(A_k)\subset B_{R_k}(X_k).
\end{eqnarray}
Combining \eqref{eq4-17} with \eqref{eq4-20} yields
\begin{eqnarray}\label{eq4-21}
	M_k(X)\leq 2 M_k(A_k) ,  \,\, X\in B_{R\lambda_k}(A_k).
\end{eqnarray}

Again Let $\bar{R} = R/\sqrt{n+1}$ and denote $A_k = (\bar{x}_k, \bar{t}_k)$,  then it is obvious that the parabolic cylinder
$$ Q_{\bar{R} \lambda_k}(\bar{x}_k, \bar{t}_k)  = \{(x,t) \mid |x-\bar{x}_k| < \bar{R} \lambda_k, \;\; |t-\bar{t}|< \bar{R}^2 \lambda_k^2 \}$$ is contained in
$B_{R\lambda_k}(A_k)$ for sufficiently large $k$.

 We now re-scale the solutions as
 $$
   v_k(x,t)=\frac{1}{M^{\frac{2s}{p-1}}_k(\bar{x}_k, \bar{t}_k)}u_k(\lambda_kx+\bar{x}_k, \lambda_k^2 t + \bar{t}_k), \,\, (x,t) \in Q_{\bar{R}}(0,0).
 $$

Then $v_k(x,t)$ satisfies the following equation
\begin{eqnarray}\label{eq4.2}
\begin{aligned}
&(\partial_t-\Delta)^s v_k(x,t)\\
=&\lambda_k^{\frac{2sp-(2s+p-1)q}{p-1}} b(\lambda_k x+\bar{x}_k)|\nabla_x v_k|^q(x,t)
+\lambda_k^{\frac{2sp}{p-1}} f(\lambda_k x+\bar{x}_k, \lambda_k^{-\frac{2s}{p-1}} v_k(x,t))\\
:=&F_k(x, v_k(x,t), \nabla_x v_k(x,t)),\,\,\,\,  (x,t)\in Q_{\bar{R}}(0,0).
\end{aligned}
\end{eqnarray}

Moreover, by the definition of $\lambda_k$ and \eqref{eq4-21}, one has
\begin{eqnarray}\label{eq4.3}
\left(v_k^{\frac{p-1}{2s}}+|\nabla_x v_k|^{\frac{p-1}{2s+p-1}}\right)(0,0)=\lambda_k M_k(A_k)=1
\end{eqnarray}
and
\begin{eqnarray}\label{eq4.4}
\left(v_k^{\frac{p-1}{2s}}+|\nabla_x v_k|^{\frac{p-1}{2s+p-1}}\right)(x,t)\leq 2,\,\,\,  (x,t)\in Q_{\bar{R}}(0,0).
\end{eqnarray}
It follows from \eqref{eq4.4} that $v_k$ and $|\nabla_x v_k|$ are uniformly bounded in $Q_{\bar{R}}(0,0)$.
Then  for $k$ large enough,  one can deduce
$$
\lambda_k^{\frac{2sp-(2s+p-1)q}{p-1}} b(\lambda_k x+\bar{x}_k)|\nabla_x v_k|^q(x,t)\leq C_8,\,\, \,  (x,t)\in Q_{\bar{R}}(0,0).
$$
and
\begin{eqnarray*}
\begin{aligned}
&\lambda_k^{\frac{2sp}{p-1}} f(\lambda_k x+\bar{x}_k, \lambda_k^{-\frac{2s}{p-1}} v_k(x,t))\\
\leq &C_0 \lambda_k^{\frac{2sp}{p-1}}\left(1+\lambda_k^{-\frac{2sp}{p-1}}(v_k(x,t))^p\right)\\
\leq &C_0 \lambda_k^{\frac{2sp}{p-1}}+C_0(v_k(x,t))^p\\
\leq & C_9,
\end{aligned}
\end{eqnarray*}
where we have used the assumptions (f1), (f2),  and $0<q<\frac{2sp}{2s+p-1}.$
As a result, there exists a constant $C_{10}>0$ independent of $k$ such that
$$
0\leq F_k(x, v_k(x,t), \nabla_x v_k(x,t))\leq C_{10},\,\, \, (x,t)\in Q_{\bar{R}}(0,0).
$$

 It follows
from Theorem \ref{mthmu1} that there exists a positive constant $C_{11}$ independent of $k$ such that
$$
\|v_k\|_{C_{x,t}^{2s-\varepsilon, s}(Q_{\bar{R}/2}(0,0))} \leq C_{11}
$$
for small $\varepsilon>0.$
Therefore  $F_k(x, v_k(x,t), \nabla_x v_k(x,t))$ in \eqref{eq4.2} is H\"{o}lder continuous for $(x,t) \in Q_{\bar{R}/2}(0,0)$  due to $s>\frac{1}{2}.$  By virtue of Theorem \ref{mthmu2}, there exist some  positive constant $\varepsilon$ and $C_{12}$ which are independent of $k$ such that
\begin{eqnarray}\label{eq4.5}
\|v_k\|_{C_{x,t}^{2s+ \varepsilon, s+\varepsilon}(Q_{\bar{R}/4}(0,0))}\leq C_{12}.
\end{eqnarray}

The above uniform regularity estimates holds in $Q_{\bar{R}/4}(0,0)$ for any $\bar{R}$. Let $\bar{R} \to \infty$, we conclude there exists
 a converging subsequence of $\{v_k(x,t)\}$ (stilled denoted by $\{v_k(x,t)\}$) that converges point-wise in
 $\mathbb{R}^n \times \R$ to a function $v(x,t).$

Moreover, $\{v_k(x,t)\}$ converges locally in $C_{x,t}^{2s+\varepsilon, s + \varepsilon}(\R^n \times \R)$ and hence  $(\partial_t -\Delta)^s v_k$ converges point-wise in $\mathbb{R}^n \times \R$.

Applying a result in \cite{CGL} we obtain
\begin{eqnarray}\label{eq4-13a}
\lim_{k \to \infty} (\partial_t -\Delta)^s v_k(x,t)= (\partial_t-\Delta)^s v(x,t)-b, \;\; \forall \, (x,t) \in \mathbb{R}^n \times \R
\end{eqnarray}
with some nonnegative constant $b$.

Taking limit in \eqref{eq4.3}, we obtain
\begin{eqnarray}\label{eq4.6}
v^{\frac{p-1}{2s}}(0,0)+|\nabla_x v|^{\frac{p-1}{2s+p-1}}(0,0)= 1.
\end{eqnarray}

To derive a contradiction, one needs to derive the limit equation satisfied by $v$.

Since $0<q<\frac{2sp}{2s+p-1}$, $v_k$ and $|\nabla_x v_k|$ are uniformly bounded in $Q_{\bar{R}/2} (0,0),$ hence
$$
\lambda_k^{\frac{2sp-(2s+p-1)q}{p-1}} b(\lambda_k x+a_k)|\nabla_x v_k|^q(x,t)
\to 0\,\, \mbox{ as }\,\, k\to \infty,
$$

While similar to the argument as in the previous subsection, one can derive that
\begin{equation}\label{eq4.8}
\lambda_k^{\frac{2sp}{p-1}} f(\lambda_k x+a_k, \lambda_k^{-\frac{2s}{p-1}} v_k(x,t))
\to K(\bar x) v^p(x,t)\,\, \mbox{ as }\,\, k\to \infty.
\end{equation}

Consequently, combining \eqref{eq4-13a} and \eqref{eq4.8}, we deduce that $v$ is a solution of
\begin{eqnarray}\label{eq4.9}
(\partial_t-\Delta)^s v(x,t)=K(\bar x) v^p(x,t) + b,\,\, (x,t) \in \mathbb{R}^n \times \R.
\end{eqnarray}
for some nonnegative constant $b$. By a similar argument as in the previous section, one drives that
$b$ must be zero.

By virtue of \eqref{eq4.6} and  the maximum principle,
one can deduce that
$$v(x,t)>0,\,\, \forall (x,t) \in \mathbb{R}^n \times \R.$$

Now we arrive at a bounded positive solution of
$$(\partial_t-\Delta)^s v(x,t)=K(\bar x) v^p(x,t).\,\, (x,t) \in \mathbb{R}^n \times \R.$$
This again contradicts the non-existence results in \cite{FP} (Theorem 1.1) if $1< p < \frac{n+2}{n+2-2s}$ and hence completes
the proof of Theorem \ref{thm4.2}.
\end{proof}

{\bf{Declaration:}} Chen is
 partially supported by MPS Simons Foundation 847690.

 Guo is partially supported by the National Natural Science Foundation of China (Grant No.12501145), the Natural Science Foundation of Shanghai (No.25ZR1402207),   the China Postdoctoral Science Foundation (No.2025T180838 and 2025M773061), the Postdoctoral Fellowship Program of CPSF (No. GZC20252004), and the Institute of Modern Analysis-A Frontier Research Center of Shanghai.

 Li is  partially supported by the National Natural Science Foundation of China (Grant No. W2531006, 12250710674, 12031012 and 11831003) and the Institute of Modern Analysis-A Frontier Research Center of Shanghai.
 \medskip

{\bf{Date availability statement:}} No data was used for the research described in the article.
\medskip

{\bf{Conflict of interest statement:}} There is no conflict of interest.

\bigskip

Wenxiong Chen

Department of Mathematical Sciences

Yeshiva University

New York, NY, 10033, USA

wchen@yu.edu
\medskip

Yahong Guo

School of Mathematical Sciences

Shanghai Jiao Tong University

Shanghai, 200240, P.R. China

yhguo@sjtu.edu.cn
\medskip

Congming Li

School of Mathematical Sciences

Shanghai Jiao Tong University

Shanghai, 200240, P.R. China

congming.li@sjtu.edu.cn

\end{document}